\documentclass[12pt]{amsart}
\usepackage[sans]{dsfont}
\usepackage{amsfonts,amssymb,amsbsy,amsmath,amsthm,enumerate}

\numberwithin{equation}{section}

\newtheorem{theorem}{Theorem}[section]
\newtheorem{proposition}[theorem]{Proposition}
\newtheorem{lemma}[theorem]{Lemma}
\newtheorem{follow}[theorem]{Corollary}

\newtheorem{pr}[theorem]{Proposition}
\theoremstyle{definition}
\newtheorem{remark}[theorem]{Remark}
\newtheorem*{remark*}{Remark}

\def\beq{\begin{equation}}
\def\eeq{\end{equation}}
\newcommand{\bea}{\begin{eqnarray}}
\newcommand{\eea}{\end{eqnarray}}
\newcommand{\beas}{\begin{eqnarray*}}
\newcommand{\eeas}{\end{eqnarray*}}

\DeclareMathOperator{\Tr}{Tr}

\DeclareMathOperator{\Ker}{Ker}
\DeclareMathOperator{\sign}{sign}
\DeclareMathOperator{\Op}{{Op}^{\it w}}
\DeclareMathOperator{\Lag}{L}
\DeclareMathOperator*{\wlim}{w-lim}

\newcommand{\one}{\mathds{1}}
\newcommand{\bel}{\begin{equation} \label}
\newcommand{\ee}{\end{equation}}

\newcommand{\bx}{{\bf x}}
\newcommand{\balpha}{\boldsymbol{\alpha}}
\newcommand{\bbeta}{\boldsymbol{\beta}}
\newcommand{\btheta}{\boldsymbol{\theta}}
\newcommand{\bzeta}{\boldsymbol{\zeta}}
\newcommand{\bxi}{\boldsymbol{\xi}}
\newcommand{\bomega}{\boldsymbol{\omega}}
\newcommand{\supp}{{\text{supp}}}
\newcommand{\rd}{{\mathbb R}^{2}}
\newcommand{\re}{{\mathbb R}}
\newcommand{\ze}{{\mathbb Z}}
\newcommand{\R}{{\mathbb R}}
\newcommand{\N}{{\mathbb N}}

\newcommand{\pdo}{$\Psi$DO }
\newcommand{\pdos}{$\Psi$DOs }

\newcommand{\abs}[1]{\lvert#1\rvert}
\newcommand{\Abs}[1]{\left\lvert#1\right\rvert}
\newcommand{\norm}[1]{\lVert#1\rVert}
\newcommand{\jap}[1]{\langle#1\rangle}
\newcommand{\mathcalU}{{\mathcal U}}
\newcommand{\bz}{\mathbf{z}}

\begin{document}
\title[Eigenvalue clusters]{Asymptotic Density of Eigenvalue Clusters for the Perturbed Landau
Hamiltonian}

\author[A.~Pushnitski]{Alexander Pushnitski}
\author[G.~Raikov]{Georgi Raikov}
\author[C.~Villegas-Blas]{Carlos Villegas-Blas}

\begin{abstract}
We consider the Landau Hamiltonian (i.e. the 2D Schr\"odinger
operator with  constant  magnetic field) perturbed by an electric
potential $V$ which decays sufficiently fast at infinity. The
spectrum of the perturbed Hamiltonian consists of clusters of
eigenvalues which accumulate to the Landau levels. Applying a
suitable version of the anti-Wick quantization, we investigate the
asymptotic distribution of the eigenvalues within a given cluster
as the number of the cluster tends to infinity. We obtain an
explicit description of the asymptotic density of the eigenvalues
in terms of the Radon transform of the perturbation potential $V$.
\end{abstract}

\maketitle

{\bf Keywords}: perturbed Landau Hamiltonian, asymptotic density for eigenvalue clusters,
anti-Wick quantization, Radon transform \\

{\bf  2010 AMS Mathematics Subject Classification}:  35P20, 35J10,
 47G30, 81Q10\\

\section{Introduction and main results}\label{intro}

\subsection{Introduction}

Let
    $$
    H_0 : = \left(-i\frac{\partial}{\partial x}+\frac{B}{2}y\right)^2 +
\left(-i\frac{\partial}{\partial y}-\frac{B}{2}x\right)^2,
    $$
    be the self-adjoint operator defined initially on
    $C_0^{\infty}(\rd)$, and then closed in $L^2(\rd)$. The
    operator
 $H_0$ is the Hamiltonian of a non-relativistic spinless 2D quantum
 particle subject to a constant magnetic field of strength $B > 0$.
It is often called the Landau Hamiltonian in honor of the author
of the pioneering paper \cite{lan}.
    The spectrum of $H_0$ consists of the eigenvalues (called Landau levels)
    $\lambda_q =B (2q+1)$, $q \in {\mathbb Z}_+ = 0,1,2,\ldots$. The multiplicity of each of these
    eigenvalues is infinite, and so
$$
\sigma(H_0)=\sigma_\text{ess}(H_0)= \bigcup_{q=0}^{\infty}\{\lambda_q\},
 \quad
 \lambda_q=B(2q+1).
 $$

Next, let $V \in C(\rd; \re)$ satisfy the estimate
\begin{equation}
\abs{V({\bf x})}\leq C\jap{\mathbf x}^{-\rho}, \quad {\mathbf x} \in \rd,
\quad \rho>1,
\label{rho}
\end{equation}
 where  $\jap{\mathbf x} : =(1+\abs{\mathbf x}^2)^{1/2}$.
We also denote by $V$ the operator of multiplication by $V$
    in $L^2(\rd)$. Consider the perturbed Landau Hamiltonian $H =
    H_0 + V$. The spectrum of $H$ consists of eigenvalue clusters
    around the Landau levels. More precisely,
    we have $\sigma_\text{ess}(H)=\sigma_\text{ess}(H_0)$ and so
all eigenvalues of $H$ in $\re \setminus \sigma_{\rm
    ess}(H)$ have finite multiplicities and can accumulate only to the
    Landau levels $\lambda_q$.
   Our first preliminary result says that the eigenvalue
    clusters shrink towards the Landau
    levels as $O(q^{-1/2})$ for $q\to\infty$:
\begin{proposition}\label{p11}
Assume \eqref{rho};
then there exists $C_1>0$  such that for all
$q \in {\mathbb Z}_+$ one has
\begin{equation}
\sigma(H) \cap [\lambda_q -B, \lambda_q + B]\subset
(\lambda_q - {C_1}{\lambda}_q^{-1/2},
\lambda_q + {C_1}{\lambda}_q^{-1/2}).
\label{a7}
\end{equation}
\end{proposition}

The proof is given in Section~\ref{s3}.

\begin{remark} \label{r1}
\begin{enumerate}[(i)]
\item Obviously, the above estimate $O(\lambda_q^{-1/2})$ for the
width of the $q$th cluster can also be written as $O(q^{-1/2})$;
however, as we will see, $\lambda_q$ provides a more natural scale
than $q$. \item Simple considerations (see Remark~3.2 in
\cite{korpu}) show that the estimate $O(\lambda_q^{-1/2})$ cannot
be improved: the eigenvalue clusters have width $\geq
c\lambda_q^{-1/2}$ with $c>0$ (unless $V\equiv0$). This will also
follow from the main result of this paper. \item
Proposition~\ref{p11} was proven in \cite{korpu} for  $V\in
C_0^\infty(\rd)$. The proof we give here not only covers the case
of more general potentials $V$, but also is based on different
ideas than those of \cite{korpu}.
\end{enumerate}
\end{remark}

\subsection{Main result}
Our purpose is to describe the asymptotic density of eigenvalues
in the $q$th cluster as $q \to \infty$. Let $\one_{\mathcal O}$
denote the characteristic function of the set $\mathcal{O} \subset
\re$. For $q \in {\mathbb Z}_+$ and $\mathcal{O} \in {\mathcal
B}(\re)$, the Borel $\sigma$-algebra on $\re$, set
$$
\mu_q(\mathcal{O}) : = {\rm
rank}\,\one_{\lambda_q^{-1/2}\mathcal{O} + \lambda_q}(H).
$$
The measure $\mu_q$ is not finite, and not even $\sigma$-finite,
but if $\mathcal{O}$ is bounded, and its closure does not contain
the origin, we have $\mu_q(\mathcal{O}) < \infty$ for $q$ sufficiently large.
In particular, for any fixed bounded interval
$[\alpha,\beta]\subset \re\setminus\{0\}$ we have
\begin{equation}
\mu_q([\alpha,\beta]) = \sum_{\lambda_q+\alpha\lambda_q^{-1/2}\leq
\lambda\leq \lambda_q+\beta\lambda_q^{-1/2}} \dim\Ker(H-\lambda I)
< \infty
\end{equation}
 for all sufficiently
large $q$. Below we study the asymptotics of the counting measure
$\mu_q$ as $q\to\infty$. In order to describe the limiting
measure, we need to fix some notation. We denote by ${\mathbb T}
\subset \rd$ the circle of radius one, centered at the origin. The
circle ${\mathbb T}$ is endowed with the usual Lebesgue measure
normalized so that $\int_{{\mathbb T}} d \omega = 2\pi$. For
$\omega=(\omega_1,\omega_2)\in\mathbb T$, we denote
$\omega^\perp=(-\omega_2,\omega_1)$. We set
$$
\widetilde{V}(\omega, b)= \frac{1}{2 \pi} \int^{\infty}_{-\infty}
V(b\omega +t\omega^\perp)dt, \quad \omega \in {\mathbb T}, \quad b
\in\R.
$$
Thus, $\widetilde{V}$ is (up to a factor) the Radon transform of
$V$.
In order to make our notation more concise, we find convenient to introduce
the Banach space  $X_\rho$ of all potentials $V\in C(\re^2,\re)$
that satisfy \eqref{rho} equipped with the norm
\begin{equation}
\norm{V}_{X_\rho} = \sup_{\mathbf x\in\re^2} \jap{\mathbf x}^\rho \abs{V(\mathbf x)}.
\label{xrho}
\end{equation}
Using this notation, by
an elementary calculation one finds
\begin{equation}
\abs{\widetilde{V}(\omega,b)}\leq  C_\rho \norm{V}_{X_\rho} \jap{b}^{1-\rho},
\quad b\in\R.
\label{a4}
\end{equation}
Define the measure $\mu$  by
$$
\mu(\mathcal{O})= \frac{1}{2\pi} \, \left|
\widetilde{V}^{-1}(B^{-1}\mathcal{O})\right|, \quad {\mathcal O} \in
{\mathcal B}(\re),
$$
where $|\cdot|$ stands for the  Lebesgue measure (on ${\mathbb T}
\times \re$).
Evidently, for any bounded interval $[\alpha, \beta]
\subset \re\setminus \{0\}$ we have $\mu([\alpha, \beta]) <
\infty$. Moreover, estimate \eqref{a4} implies that
 $\mu$ has a bounded support in
$\re$,  and
\begin{equation}
\int_\re \abs{t}^\ell d\mu(t)<\infty, \quad \forall
\; \ell>1/(\rho-1). \label{a9}
\end{equation}
 Our main result is:
\begin{theorem}\label{th12}
Let $V\in C(\rd)$ be a continuous function that satisfies
\eqref{rho}. Then, for any function $\varrho \in C^{\infty}_0(\re
\setminus\{0\})$, we have
\begin{equation}
\lim_{q \to \infty} \lambda_q^{-1/2} \int_\R
\varrho(\lambda)d\mu_q(\lambda) = \int_\R
\varrho(\lambda)d\mu(\lambda). \label{11a}
\end{equation}
\end{theorem}

\begin{remark} \label{r2a}
\begin{enumerate}[(i)]
\item
The asymptotics \eqref{11a} can be more explicitly written as
\begin{equation}
\lim_{q \to \infty} \lambda_q^{-1/2}
\Tr \varrho(\sqrt \lambda_q (H-\lambda_q))
=
\frac1{2\pi} \int_{\mathbb T} \, \int_{\R} \,
\varrho(B \widetilde{V}(\omega,b)) \, db \, d\omega. \label{11}
\end{equation}
\item By  standard approximation arguments,
the asymptotics \eqref{11a} can be extended to a wider class of continuous functions $\varrho$.
Further,
it follows from
Theorem~\ref{th12} that if $[\alpha, \beta] \subset \re\setminus
\{0\}$, and $\mu(\{\alpha\})=\mu(\{\beta\})=0$, then
$$
\lim_{q \to \infty}\lambda_q^{-1/2}\mu_q([\alpha,\beta]) =
\mu([\alpha,\beta]).
$$
However, the assumption $\mu(\{\alpha\})=\mu(\{\beta\})=0$ does
not automatically hold, i.e. in general the measure $\mu$ may have
atoms. Indeed, a description of the class of all Radon transforms
$\widetilde{V}(\omega,b)$ of functions $V\in C_0^\infty(\R^2)$ is
well known, see e.g. \cite[Theorem~2.10]{helg}. According to this
description, if $a\in C_0^\infty(\re)$ is an even real-valued
function, then
$\widetilde{V}(\omega,b) : =a(b)$, $b \in \re$, $\omega \in {\mathbb T}$, is a Radon transform
in this class. Of course, if the derivative $a'(b)$ vanishes on
some open interval, then the corresponding measure $\mu$ has an
atom.

\item
If $V\in C_0^\infty(\rd)$, one can prove (see \cite{korpu})
that the trace in the l.h.s. of \eqref{11} has a complete asymptotic
expansion in inverse powers of $\lambda_q^{1/2}$,
but the formulae for the higher order coefficients
of this expansion are not known.\\
\end{enumerate}
\end{remark}

\subsection{Method of proof}
 Let $P_q$ be the orthogonal projection in $L^2(\rd)$ onto the subspace
${\rm Ker}\, (H_0-\lambda_qI)$.
For $\ell\geq1$, let $S_\ell$ be the Schatten-von Neumann class,
with the norm $\norm{\cdot}_\ell$; the usual operator norm is
denoted by $\norm{\cdot}$.

We first fix a natural number $\ell$ and examine the asymptotics of the trace
in the l.h.s. of \eqref{11}
for functions $\varrho\in C_0^\infty(\R)$ such that
$\varrho(\lambda)=\lambda^\ell$ for small $\lambda$.
We have the following
(fairly standard) technical result:
\begin{lemma}\label{l31}
For any real $\ell>1/(\rho-1)$, the operators
$(H-\lambda_q)^\ell \one_{(\lambda_q - B, \lambda_q+B)}(H)$
and $(P_q V P_q)^\ell$ belong to the trace class and
\bel{31}
\Tr\{(H-\lambda_q)^\ell
\one_{(\lambda_q - B,\lambda_q+B)}(H)\}
=
\Tr
(P_q V P_q)^\ell + o(\lambda_q^{-(\ell-1)/2} ), \quad q \to \infty.
\ee
\end{lemma}
The proof of this
lemma is given in Subsection~\ref{s3c}.
This lemma essentially reduces the question to the study of the asymptotics
of traces of $(P_q V P_q)^\ell$.
Our main technical result is the
following statement:
\begin{theorem}\label{th13}
Let $V$ satisfy \eqref{rho} and let $B_0>0$.
\begin{enumerate}[\rm (i)]
\item For some $C=C(B_0)$, one has
\begin{equation}
\sup_{q\geq0} \sup_{B\geq B_0} \lambda_q^{1/2}B^{-1}\norm{P_q VP_q}
\leq
C\norm{V}_{X_\rho}.
\label{a6}
\end{equation}
\item
For any real $\ell>1/(\rho-1)$, we have  $P_qVP_q\in S_\ell$, and
for some $C=C(B_0,\ell)$, the estimate
\begin{equation}
\sup_{q\geq0} \sup_{B\geq B_0} \lambda_q^{(\ell-1)/(2\ell)} B^{-1}\norm{P_qVP_q}_{\ell}
\leq C \norm{V}_{X_\rho}
\label{b1}
\end{equation}
holds true.
\item For any integer $\ell>1/(\rho-1)$, we have \bel{12}
\lim_{q\to \infty} \lambda_q^{(\ell-1)/2} \Tr(P_q VP_q)^\ell
=
\frac{B^\ell}{2\pi} \int_{{\mathbb T}} \, \int_{\R}
\;\widetilde{V}(\omega,b)^\ell \, db \, d\omega. \ee
\end{enumerate}
\end{theorem}
Although in our main Theorem~\ref{th12} the strength $B$ of the magnetic
field is assumed to be fixed, we make the dependence on $B$
explicit in the estimates \eqref{a6}, \eqref{b1}, as these estimates
are of an  independent  interest  (see e.g. \cite{dp}), and can be used
in the study of other asymptotic regimes.

The proof of Theorem~\ref{th13} consists of two steps. In
Section~\ref{s4} we establish the unitary equivalence of the
Berezin-Toeplitz operator $P_qVP_q$ to a certain generalized
anti-Wick pseudodifferential operator ($\Psi$DO) whose symbol
$V_B$ is defined explicitly below in \eqref{sof30}. Further, in
Section~\ref{sec.c} we study this $\Psi$DO, prove appropriate
estimates, and analyze its  asymptotic behavior as $q \to \infty$.

A combination of
\eqref{31} and \eqref{12} essentially yields
\eqref{11a} for a function $\varrho\in C_0^\infty(\R)$ such that
$\varrho(\lambda)=\lambda^\ell$ for small $\lambda$. After this,
the main result follows by an application of the Weierstrass'
approximation theorem; this argument is given in
Subsection~\ref{s32}.

\begin{remark}\label{rmk}
In \cite{korpu} the limit \eqref{12} was computed for $\ell=1,2$,
but the result was written in a form not suggestive of the general
formula.
\end{remark}

\subsection{Semiclassical interpretation}
Consider the classical Hamiltonian function
    \bel{gr1}
{\mathcal H}(\bxi,\bx) = (\xi+\tfrac12 B y)^2+(\eta -\tfrac12
Bx)^2, \quad \bxi : = (\xi,\eta) \in \rd, \quad \bx : = (x,y) \in
\rd,
    \ee
in the phase space $T^* \rd = \re^4$ with the standard symplectic
form. The projections onto the configuration space of the orbits
of the Hamiltonian flow of ${\mathcal H}$ are circles of radius
$\sqrt{E}/B$, where $E>0$ is the value of the energy corresponding
to the orbit. The classical particles move around these circles
with  period $T_B=\pi/B$. The set of these orbits can be
parameterized by the energy $E>0$ and the center $\mathbf
c\in\re^2$ of a circle. Let us denote the path in the
configuration space corresponding to such an orbit by
$\gamma(\mathbf c,E,t)$, $t\in[0,T_B)$, and set
\begin{equation}
\langle V \rangle (\mathbf c,E)  = \frac1{T_B} \int_0^{T_B}
V(\gamma(\mathbf c,E,t))dt, \quad T_B=\pi/B. \label{a1}
\end{equation}

For an energy $E>0$, consider the set $M_E$ of all orbits with
this energy. The set $M_E$ is a smooth manifold with coordinates
$\mathbf c\in\re^2$. It can be considered as the quotient of the
constant energy surface
$$
\Sigma_E = \{(\bxi,\bx) \in \re^4 \mid {\mathcal H}(\bxi,\bx)=E\}
$$
with respect to the flow of ${\mathcal H}$. Restricting the
standard Lebesgue measure of $\re^4$ to $\Sigma_E$ and then taking
the quotient, we obtain the measure $B\, dc_1\, dc_2$ on $M_E$. An
elementary calculation shows that the r.h.s. of \eqref{11} can be
rewritten as
\begin{equation}
\frac1{2\pi} \int_{\mathbb T}  \int_{\R}
\varrho(B\widetilde{V}(\omega,b)) db \, d\omega= \frac1{2\pi}
\lim_{E\to\infty} \frac1{\sqrt{E}}\int_{\re^2} \varrho(\sqrt{E}\,
\langle V \rangle (\mathbf c,E)  )  \, B\, dc_1\, dc_2. \label{a2}
\end{equation}
The basis of this calculation is the fact that as $E\to\infty$,
the radius $\sqrt{E}/B$ of the circles representing the classical
orbits tends to infinity. Thus, the classical
orbits approximate straight lines on any compact domain of the
configuration space.

Given \eqref{a2}, we can rewrite the main result as
\begin{equation}
\lim_{q \to \infty} \frac{1}{\sqrt \lambda_q} \Tr \varrho(\sqrt
\lambda_q (H-\lambda_q)) = \frac1{2\pi} \lim_{E\to\infty}
\frac1{\sqrt{E}}\int_{\re^2} \varrho(\sqrt{E}\, \langle V \rangle
(\mathbf c,E) ) \, B\, dc_1\, dc_2. \label{a3}
\end{equation}
This agrees with the semiclassical intuition. Formula \eqref{a3}
corresponds to the well known ``averaging principle'' for systems
close to integrable ones. This principle states that a good
approximation is obtained if one replaces the original
perturbation by the one which results by averaging the original
perturbation along the orbits of the free dynamics. This method is
very old; quoting V.~Arnold \cite[Section 52]{arn}: ``In studying
the perturbations of planets on one another, Gauss proposed to
distribute the mass of each planet around its orbit proportionally
to time and to replace the attraction of each planet by the
attraction of the ring so obtained''.

\subsection{Related results}
\subsubsection{Asymptotics for eigenvalue clusters for manifolds
with closed geodesics} In spectral theory, results of this type
originate from the classical work by A.~Weinstein \cite{wein} (see
also \cite{CdV}). Weinstein considered
the operator $-\Delta_{\mathcal M} + V$, where $\Delta_{\mathcal
M}$ is the Laplace-Beltrami operator on a compact Riemannian
manifold ${\mathcal M}$ with periodic bicharacteristic flow (e.g.
a sphere), and $V \in C({\mathcal M};\re)$.
In this case, all eigenvalues of $\Delta_{\mathcal M}$ have
finite multiplicities which however grow with the
eigenvalue number. Adding the perturbation $V$ creates clusters of
eigenvalues. Weinstein proved that the asymptotic density of
eigenvalues in these clusters can be described by the
density function obtained by  averaging $V$ along the closed geodesics on ${\mathcal M}$.
Let us illustrate these results with the case  ${\mathcal M}
 = {\mathbb S}^2$. It is well known that the eigenvalues of $-\Delta_{{\mathbb
 S}^2}$ are $\lambda_{q} = q (q + 1)$, $q \in {\mathbb
 Z}_+$, and their multiplicities are $d_q = 2
 q + 1$. For $V \in C({\mathbb S}^2; \re)$ set
 $$
 \widetilde{V}(\omega) : = \frac{1}{2\pi} \int_0^{2\pi} V({\mathcal
 C}_\omega(s)) ds,  \quad \omega \in {\mathbb
 S}^2,
 $$
 where ${\mathcal C}_\omega(s) \in {\mathbb S}^2$ is the great
 circle orthogonal to $\omega$, and $s$ is the arc length on this
 circle. Then for each $\varrho \in C_0^{\infty}({\mathbb S}^2;
 \re)$ we have
    \bel{end20}
 \lim_{q \to \infty}\frac{\Tr\varrho(-\Delta_{{\mathbb
 S}^2} + V - \lambda_q)}{d_q} = \int_{{\mathbb
 S}^2} \varrho(\widetilde{V}(\omega)) dS(\omega)
    \ee
 where $dS$ is the normalized Lebesgue measure on ${\mathbb
 S}^2$. Since  ${\mathbb S}^2$ can be identified
 with its set of oriented geodesics ${\mathcal G}$, the r.h.s. of \eqref{end20} can be interpreted as an integral
 with respect to the $SO(3)$--invariant normalized measure on
 ${\mathcal G}$.
 This result admits  extensions to the case ${\mathcal M} =
 {\mathbb S}^n$ with $n>2$, and, more generally, to the case where
 ${\mathcal M}$ is a compact symmetric manifold of rank 1 (see
 \cite{wein, CdV}).\\

In more recent works \cite{tvb, vb1, urvb}, the relation between
the quantum Hamiltonian of the hydrogen atom and the
Laplace-Beltrami operator on the unit sphere is exploited, and the
asymptotic distribution within the eigenvalue clusters of the
perturbed Hamiltonian hydrogen atom is investigated. The
asymptotic density of eigenvalues in these clusters was described
in terms of the perturbation averaged along the trajectories of
the unperturbed dynamics (i.e. the solutions to the Kepler
problem).

Among the main technical tools used in \cite{tvb, vb1, urvb} which
originate from \cite{gp, tw}, are the Bargmann-type
representations of the particular quantum Hamiltonians considered,
implemented via  the  so-called Segal-Bargmann  transforms. In our
analysis generalized coherent states and associated anti-Wick
\pdos  closely related to the Bargmann representation and the
Segal-Bargmann transform appear again in a natural way (see
Section~\ref{s4}), although their role is different from the one
of their counterparts in  \cite{tvb, vb1, urvb}.

Although this paper is inspired by \cite{wein,tvb, vb1,urvb}, much
of our construction (see Section~\ref{s3}) is based on the
analysis of \cite{korpu}.
    In \cite{korpu} it was proven that for $V\in C_0^\infty(\rd)$ the trace
    in the l.h.s. of \eqref{12} has
    complete asymptotic expansions
    in inverse powers of $\lambda_q^{1/2}$.
    However, the coefficients of this expansion have
    not been computed  explicitly; see Remark~\ref{rmk} above.

\subsubsection{Strong magnetic field asymptotics} It is
useful to compare our main result with the asymptotics as $B
\to \infty$ of the eigenvalues of $H$. It has been found in
\cite{r1} (see also \cite{ivrii}) that
    \bel{gr50a}
    \lim_{B \to \infty} B^{-1} \Tr \varrho(H-\lambda_q) = \frac{1}{2\pi} \int_{\rd} \varrho(V(\bx)) d\bx =
    \int_\re \varrho(t) dm(t)
        \ee
        where $\varrho \in C_0^{\infty}(\re\setminus \{0\})$, $q \in {\mathbb Z}_+$, $V \in L^p(\rd)$, $p>1$,
        and $m(\mathcal{O}) : = \frac{1}{2\pi} \left|V^{-1}(\mathcal{O})\right|$,
        $\mathcal{O} \in {\mathcal B}(\re)$. Similarly to Theorem \ref{th12},
        the proof of \eqref{gr50a} is based on an analogue of Theorem \ref{th13} (i) -- (ii)
        (see Lemma \ref{l21} below), and the asymptotic relations
        \bel{gr51a}
    \lim_{B \to \infty} B^{-1} \Tr (P_q V P_q)^\ell = \lim_{B \to \infty} B^{-1}
    \Tr P_q V^\ell P_q  = \frac{1}{2\pi} \int_{\rd} V(\bx)^{\ell} d\bx, \quad q \in {\mathbb Z}_+,
        \ee
        with $V \in C_0^\infty(\rd)$ and $\ell \in \mathbb{N}$, close in spirit to \eqref{12}.
        Since $B m([\alpha, \beta]) = \frac{1}{2\pi} \left|V_B^{-1}([\alpha, \beta])\right|$,
        $[\alpha, \beta] \subset \re \setminus \{0\}$, where $V_B$ (see \eqref{sof30})
        is the symbol of the generalized anti-Wick \pdo to which $P_q V P_q$ is unitarily equivalent,
        we find that \eqref{gr50a} is again a result of semiclassical nature. However, \eqref{gr51a}
        implies that in the strong magnetic field regime  the main asymptotic terms of $\Tr (P_q V P_q)^\ell$ and $\Tr P_q V^\ell P_q$ coincide, and hence in the first approximation the  commutators $[V, P_q]$
        are negligible, while \eqref{12} shows
        that obviously this is not the case in the high energy  regime considered in
        the present article. Hence, Theorem~\ref{th12} retains ``more quantum flavor" than \eqref{gr50a},
        and hence its proof is technically much more involved.

\subsubsection{The spectral density of the scattering matrix for high
energies} In the recent work \cite{bupu} inspired by this paper,
D. Bulger and A. Pushnitski considered the scattering matrix
$S(\lambda)$, $\lambda>0$, for the operator pair $(-\Delta + V,
-\Delta)$ where $\Delta$ is the standard Laplacian acting in
$L^2(\re^d)$, $d\geq2$, and $V \in C(\re^d;\re)$ is an electric
potential which satisfies an estimate analogous to \eqref{rho}.
Although the methods applied in \cite{bupu} are different from
ours, it turned out that the asymptotics as $\lambda \to \infty$
of the eigenvalue clusters for $S(\lambda)$ are written in terms
of the $X$-ray transform of $V$ in a manner similar to
\eqref{11a}.

\section{Unitary equivalence of Berezin-Toeplitz operators and generalized anti-Wick \pdos}\label{s4}

\subsection{Outline of the section}
\label{ss21} From methodological point of view, this section plays a
central role in the proof of Theorem \ref{th12}. Its principal goal
is to establish the unitary equivalence  between the
Berezin-Toeplitz operators $P_qVP_q$, $q \in {\mathbb Z}_+$, and
some generalized anti-Wick \pdos ${\rm Op}^{aw}_q(V_B)$ whose symbol
$V_B$ is defined explicitly in \eqref{sof30}. This equivalence is
proved in Theorem \ref{th.b2} below. The \pdos ${\rm Op}^{aw}_q$
introduced in Subsection \ref{ss24}, are quite similar to the
classical anti-Wick operators ${\rm Op}^{aw}_0$ (see \cite[Chapter
V, Section 2]{beshu}, \cite[Section 24]{shu}); the only difference
is that the quantization ${\rm Op}^{aw}_0$ is related to coherent
states built on the first eigenfunction $\varphi_0$ of the harmonic
oscillator \eqref{gr3}, while ${\rm Op}^{aw}_q$, $q \in {\mathbb
N}$,
is related to coherent states built on its $q$th eigenfunction $\varphi_q$. \\
In our further analysis of the operator ${\rm Op}^{aw}_q(V_B)$
performed in Section \ref{sec.c}, we also heavily use the
properties of the Weyl symbol of this operator. Thus, in
Subsections \ref{ss21} -- \ref{ss22} we introduce the Weyl
quantization ${\rm Op}^{w}$, and in Subsection \ref{ss25} we
briefly discuss its relation to ${\rm Op}^{aw}_q$. In particular,
we show that ${\rm Op}^{aw}_q(s) = {\rm Op}^{w}(s * \Psi_q)$ where
$s$ is a symbol from an appropriate class, and $2\pi \Psi_q$ is
the Wigner function associated with $\varphi_q$, defined
explicitly in \eqref{17a}. Therefore, the Berezin-Toeplitz
operator $P_qVP_q$, $q \in {\mathbb Z}_+$, with domain $P_q
L^2(\rd)$, is unitarily equivalent to  ${\rm Op}^{w}(V_B *
\Psi_q)$ (see Corollary \ref{fgr1}).

\subsection{Weyl \pdos} \label{ss22}
Let $d \geq 1$. Denote by  ${\mathcal S}(\re^d)$ the Schwartz
class, and by ${\mathcal S}'(\re^d)$ its dual class.

\begin{pr} \label{prvbp1a}
{\rm  \cite[Lemma 18.6.1]{ho}} Let  $s \in {\mathcal
S}'(\re^{2d})$. Assume that $\hat{s} \in L^1(\re^{2d})$ where
$\hat{s}$ is the Fourier transform of $s$, introduced explicitly
in \eqref{end2} below.
 Then the
operator ${\rm Op}^w(s)$ defined initially as a mapping from the
 ${\mathcal S}(\re^d)$ into  ${\mathcal
S}'(\re^d)$ by
    \bel{gr52}
\left({\rm Op}^w(s)u\right)(x) = (2\pi)^{-d} \int_{\re^d}
\int_{\re^d} s\left(\frac{x+x'}{2},\xi\right)e^{i(x-x')\cdot \xi}
u(x')
 dx'd\xi, \quad x \in \re^d,
    \ee
extends uniquely to an operator bounded in $L^2(\re^d)$. Moreover,
    \bel{end3}
\|{\rm Op}^w(s)\| \leq (2\pi)^{-d} \|\hat{s}\|_{L^1(\re^{2d})}.
    \ee
\end{pr}
Some of the arguments of our proofs require estimates which are
more sophisticated  than \eqref{end3}. Let $\Gamma(\re^{2d})$, $d\geq 1$,
denote the set of functions $s: \re^{2d} \to {\mathbb C}$ such that
    $$
    \|s\|_{\Gamma(\re^{2d})} : = \sup_{\{\alpha , \beta \in {\mathbb
Z}_+^d \; | \; |\alpha|, |\beta| \leq  [\frac{d}{2}] + 1\}}
\sup_{(x,\xi) \in \re^{2d}} |\partial_x^{\alpha}
\partial_{\xi}^{\beta} s(x,\xi)| < \infty.
    $$
    \begin{pr} \label{prvbp1}
{\rm \cite[Corollary 2.5 (i)]{abd}} There exists a constant $c_0$
such that for any $s \in \Gamma(\re^{2d})$, $d\geq 1$, we have
$$
     \|{\rm Op}^w(s)\| \leq c_0\|s\|_{\Gamma(\re^{2d})}.
    $$
\end{pr}

 Further, if $s\in
L^2(\R^{2d})$, then, obviously, the operator $\Op(s)$ belongs to the
Hilbert-Schmidt class, and
\begin{equation}
\norm{\Op(s)}_{2}^2 = \frac{1}{(2\pi)^d}\int_{\R^{2d}} \abs{s(x,\xi)}^2\,
dx \,d\xi. \label{prvb2}
\end{equation}
Next, we describe the well known metaplectic unitary equivalence
of Weyl \pdos   whose symbols are mapped into each other by a linear
symplectic change of the variables.

\begin{pr} \label{p42} {\em \cite[Chapter 7, Theorem A.2]{dsj}}
Let $\kappa: \re^{2d} \rightarrow \re^{2d}$, $d\geq 1$,  be a
linear symplectic transformation, $s_1 \in \Gamma(\re^{2d})$, and
$s_2 : = s_1 \circ \kappa$. Then there exists a unitary operator $U:
L^2(\re^d) \rightarrow L^2(\re^d)$ such that
$$
{\rm Op}^w(s_2) = U^* {\rm Op}^w(s_1) U.
$$
\end{pr}
\begin{remark} \label{r3}
\begin{enumerate}[(i)]
    \item
  The operator $U$ is called {\em the metaplectic operator} corresponding to the linear symplectic transformation $\kappa$.
There exists a one-to-one correspondence  between metaplectic
operators and linear symplectic transformations, apart from  a
constant factor of modulus 1 (see e.g. \cite[Theorem 18.5.9]{ho}). Moreover, every linear symplectic transformation $\kappa$ is a composition of a finite number of elementary linear symplectic maps (see e.g. \cite[Lemma 18.5.8]{ho}), and for each elementary linear symplectic map there exists an explicit simple  metaplectic operator (see e.g. the proof of \cite[Theorem 18.5.9]{ho}).
    \item Proposition \ref{p42} extends to a large class of not
    necessarily bounded operators. In particular, it holds for Weyl \pdos with
     quadratic symbols.
\end{enumerate}
\end{remark}

\subsection{Generalized anti-Wick \pdos}
\label{ss24}
In this subsection we introduce generalized anti-Wick \pdos.
These operators are a special case of the \pdos  with
contravariant symbols whose theory has been developed in
\cite{ber}.
Introduce the harmonic oscillator
    \bel{gr3}
h : =  - \frac{d^2}{dx^2} + x^2,
    \ee
self-adjoint in $L^2(\re)$.
It is well known
that the spectrum of $h$ is purely discrete and simple, and
consists of the eigenvalues $2q+1$, $q \in {\mathbb Z}_+$, while its associated
real-valued eigenfunctions $\varphi_q$, normalized in $L^2(\re)$, could be written as
    \bel{radi}
    {\varphi}_q(x): =
\frac{{\rm H}_{q}(x) e^{-x^2/2}}{(\sqrt{\pi}2^{q} q!)^{1/2}},
\quad x \in \re, \quad q \in {\mathbb Z}_+,
    \ee
   where
    \bel{prvb28}
{\rm H}_q(x): = (-1)^q e^{x^2/2} \left(\frac{d}{dx} - x\right)^q
e^{-x^2/2}, \quad x \in \re, \quad q \in {\mathbb Z}_+,
    \ee
 are the Hermite polynomials. Introduce the generalized coherent states (see e.g. \cite{rs})
    \bel{10d04}
    \varphi_{q;x,\xi}(y) : =  e^{i\xi
y} \varphi_q(y-x), \quad y \in \re, \quad (x,\xi) \in \rd,
    \ee
    so that $\varphi_q = \varphi_{q;0,0}$. Note that if $f \in L^2(\re)$,
then
\begin{equation} \label{17}
\|f\|^2_{L^2(\re)} = (2\pi)^{-1}\int_{\rd}|\langle f,
\varphi_{q;x,\xi}\rangle|^2dxd\xi
\end{equation}
where $\langle \cdot , \cdot\rangle$ denotes the scalar product in
$L^2(\re)$.
Introduce the orthogonal projection
    \bel{end15}
p_{q;x,\xi} : = |\varphi_{q;x,\xi}\rangle
\langle\varphi_{q;x,\xi}| : L^2(\re) \to L^2(\re), \quad q \in
{\mathbb Z}_+, \quad (x,\xi) \in \rd.
    \ee
Let $s \in L^1(\rd) + L^{\infty}(\rd)$. Define
$$
{\rm Op}_q^{aw}(s): = (2\pi)^{-1} \int_{\rd} s(x,\xi) p_{q;x,\xi}
dx d\xi
$$
as the operator generated in   $L^2(\re)$ by the bounded
sesquilinear form
    \bel{8d03} F_{q,s}[f,g] : = (2\pi)^{-1}
\int_{\rd} s(x,\xi) \langle f, \varphi_{q;x,\xi}\rangle
\overline{\langle g, \varphi_{q;x,\xi}\rangle} dx d\xi, \quad f,g
\in  L^2(\re).
    \ee
    We will call ${\rm Op}_q^{aw}(s)$ operator with {\em
anti-Wick symbol of order $q$} equal to $s$. We introduce these operators only in
the case of dimension $d=1$ since this is sufficient for our purposes; of course, their definition  extends easily to any dimension $d \geq 1$. \\
Note that the quantization ${\rm Op}_q^{aw}$ with $q = 0$
 coincides with the standard anti-Wick one (see e.g.
\cite[Section 24]{shu}). In the following lemma we summarize some
elementary basic properties of generalized anti-Wick \pdos which
follow immediately from the corresponding properties of general
contravariant \pdos.
    \begin{lemma} \label{lgr1} \cite[Section 24]{shu} \cite[Section
5.3]{beshu}
    {\em (i)} Let $s \in L^{\infty}(\rd)$. Then we have
    \bel{8d03a}
    \|{\rm Op}_q^{aw}(s)\| \leq \|s\|_{L^{\infty}(\rd)}.
    \ee
    {\em (ii)} Let $s \in L^{\ell}(\rd)$ with $\ell \in [1,\infty)$. Then we have
    \bel{8d04}
    \|{\rm Op}_q^{aw}(s)\|_\ell^\ell \leq (2\pi)^{-1}
    \|s\|^\ell_{L^\ell(\rd)}.
    \ee
    \end{lemma}

\subsection{Relation between generalized anti-Wick and Weyl \pdos}
\label{ss25}
For $q \in {\mathbb Z}_+$ set
    \bel{17a}
    \Psi_q(x,\xi) = \frac{(-1)^q}{\pi}
{\rm L}_q(2(x^2 + \xi^2)) e^{-(x^2 + \xi^2)}, \quad (x,\xi) \in
\rd,
    \ee
where
\bel{lagpol}
    {\rm L}_q(t) : = \frac{1}{q!} e^t \frac{d^q(t^q
e^{-t})}{dt^q} = \sum_{k=0}^q \binom{q}{k} \frac{(-t)^k}{k!},
\quad t \in \re, \ee
are the Laguerre polynomials.

\begin{lemma} \label{l41}
For a fixed $(x,\xi) \in \rd$ we have
$p_{q;x,\xi} = \Op (\varsigma_{q;x,\xi})$
where $p_{q;x,\xi}$ is the orthogonal projection defined in \eqref{end15}, and
    \bel{gr50}
\varsigma_{q;x,\xi}(x',\xi'): =
2\pi \Psi_q(x'-x,\xi'-\xi), \quad (x'.\xi´) \in \rd.
    \ee
\end{lemma}
\begin{proof}
Using the well-known relation between the Schwartz kernel of a
linear operator and its Weyl symbol (see e.g \cite[Eq.
(18.5.4)'']{ho}), we find that the Weyl symbol $\varsigma_{q;x,\xi}$
of the projection $p_{q;x,\xi}$ satisfies
    \bel{gr51}
    \varsigma_{q;x,\xi}(x',\xi') =
\int_\re e^{-iv\xi'}\varphi_{q;x,\xi}(x'+v/2) \overline
{\varphi_{q;x,\xi}(x'-v/2)} dv.
    \ee
    By \eqref{radi} and \eqref{10d04},
    $$
    \int_\re e^{-iv\xi'}\varphi_{q;x,\xi}(x'+v/2) \overline
{\varphi_{q;x,\xi}(x'-v/2)} dv =
    $$
    \bel{18}
\frac{1}{\sqrt{\pi}2^q q!} \int_{\re} e^{iv(\xi-\xi')} {\rm
H}_q(x' + \frac{1}{2} v - x) {\rm H}_q(x' - \frac{1}{2} v - x)
e^{-(x' + \frac{1}{2} v - x)^2/2} e^{-(x' - \frac{1}{2} v -
x)^2/2} dv.
    \ee
    Changing the variable of integration
    $v = 2(t + i(\xi -\xi'))$, and bearing of mind the parity of the Hermite polynomial
${\rm H}_q$, we get
$$
\int_{\re} e^{iv(\xi-\xi')} {\rm H}_q(x' + \frac{1}{2} v - x) {\rm
H}_q(x' - \frac{1}{2} v - x) e^{-(x' + \frac{1}{2} v - x)^2/2}
e^{-(x' - \frac{1}{2} v - x)^2/2} dv =
$$
\bel{19} 2(-1)^q e^{-(x'-x)^2 - (\xi'-\xi)^2} \int_{\re} e^{-t^2}
{\rm H}_q(t-(x-x'-i(\xi-\xi'))) {\rm H}_q(t+ x-x'+i(\xi-\xi')) dt.
\ee Employing the relation between the Laguerre polynomials and
the integrals of Hermite polynomials (see e.g. \cite[Eq.
7.377]{grry}), we obtain
$$
\int_{\re} e^{-t^2} {\rm H}_q(t-(x-x'-i(\xi-\xi'))) {\rm
H}_q(t+x-x'+i(\xi-\xi')) dt =
$$
\bel{20z} \sqrt{\pi}2^q q! {\rm L}_q(2((x'-x)^2+(\xi-\xi')^2)).
\ee Putting together \eqref{gr51} -- \eqref{20z}, we obtain
\eqref{gr50}.
\end{proof}
\begin{remark}  \label{r5} Let $\psi \in L^2(\re)$ and $\|\psi\|_{L^2(\re)} =
1$. Then the Weyl symbol of the rank-one orthogonal projection
$|\psi\rangle \langle\psi|$, is called the Wigner function
associated with $\psi$ (see e.g. \cite[Definition 2.2]{tp}). Thus,
Lemma \ref{l41} tells us, in particular, that $2\pi \Psi_q$ is the
Wigner function associated with $\varphi_q$.
    \end{remark}
 Lemma \ref{l41}
immediately entails the following
\begin{follow} \label{f42}
Let $s \in L^1(\rd) + L^{\infty}(\rd) $. Then we have
    \bel{9d05}
    {\rm Op}^{aw}_q(s) = {\rm Op}^{w}(\Psi_q * s).
    \ee
    \end{follow}
    \subsection{Metaplectic mapping of the operators $H_0, P_q$ and $V$}
\label{ss23} For ${\bf x} = (x,y) \in \rd, \;  \bxi = (\xi,\eta)
\in \rd$, set
    \bel{sof22}
     {\varkappa}_{B}({\bf x}, \bxi): =
\left(\frac{1}{\sqrt{B}} (x-\eta), \frac{1}{\sqrt{B}} (\xi-y),
\frac{\sqrt{B}}{2}(\xi+y), -\frac{\sqrt{B}}{2}(\eta+x)\right).
    \ee
Evidently, the transformation ${\varkappa}_B$ is linear and
symplectic. Define the unitary operator ${\mathcal U}_{B}:
L^2(\rd) \to L^2(\rd)$ by
    \bel{133}
    ({\mathcal U}_{B}
u)(x,y): = \frac{\sqrt{B}}{2\pi} \int_{\rd}
e^{i\phi_{B}(x,y;x',y')} u(x',y') dx'dy'
    \ee
    where
$$
\phi_{B}(x,y;x',y'): = B \frac{xy}{2} + B^{1/2}(xy' - yx') - x'y'.
$$
Writing ${\varkappa}_{B}$ as a product of elementary linear
symplectic transformations (see e.g. \cite[Lemma 18.5.8]{ho}), we
can easily  check that
 ${\mathcal U}_B$ is a metaplectic operator corresponding to
 ${\varkappa}_B$. Note that
 $$
 {\mathcal H}\circ\varkappa_B(\bx, \bxi) = B(\xi^2 + x^2), \quad \bx = (x,y) \in \rd, \quad \bxi = (\xi, \eta) \in \rd,
 $$
 where ${\mathcal H}$ is the Weyl symbol of the operator $H_0$ defined in \eqref{gr1}.
 On the other hand, $B(\xi^2 + x^2)$
 is the Weyl
 symbol  of the
 operator $B(h \otimes I_y)$ self-adjoint in $L^2(\rd_{x,y})$
 where  $h$ is
the harmonic oscillator \eqref{gr3},
    acting in $L^2(\re_x)$, and
$I_y$ is the identity operator in $L^2(\re_y)$.  Denote by $p_q =
|\varphi_q\rangle\langle\varphi_q| = p_{q;0,0}$ the orthogonal
projection onto ${\rm Ker}\,(h-2q-1)$, $q \in {\mathbb Z}_+$.
Applying Proposition \ref{p42} with $\kappa= \varkappa_B$, and
bearing in mind Remark \ref{r3} (ii), we obtain the following
\begin{follow} \label{f41}
(i) We have
    \bel{sof20}
{\mathcal U}_{B}^* H_0 {\mathcal U}_{B} = B\left(h \otimes
I_y\right),
    \ee
 \bel{sof21}
{\mathcal U}_{B}^* P_q {\mathcal U}_{B} = p_q \otimes I_y, \quad q
\in {\mathbb Z}_+.
    \ee
(ii) If $V \in \Gamma(\rd)$, then
    \bel{sof32}
    {\mathcal U}_{B}^* V
{\mathcal U}_{B} = {\rm Op}^w({\bf V}_B)
    \ee
where
    \bel{8d01}
    {\bf V}_B (x,y;\xi,\eta) : = V(B^{-1/2}(x-\eta),
B^{-1/2}(\xi-y)), \quad (x,y;\xi,\eta) \in \re^4.
    \ee
\end{follow}
\begin{remark} \label{r2}
Various versions of the symplectic transformation $\varkappa_B$ in
\eqref{sof22} and the corresponding metaplectic operator
${\mathcal U}_B$ in \eqref{133} have been used in the spectral
theory of the perturbations of the Landau Hamiltonian (see e.g.
\cite{helsjo}). Of course, the close relation between the Landau
Hamiltonian $H_0$ and the harmonic oscillator $h$ is well-known
since the seminal work \cite{lan} where the basic
spectral properties of $H_0$ were first described.\\
\end{remark}

\subsection{Unitary equivalence of $P_q V P_q$ and ${\rm Op}_q^{aw}(V_B)$}
\label{ss26}
Set
    \bel{sof30}
    V_B(x,y)=V(-B^{-1/2}y,-B^{-1/2}x), \quad (x,y)\in\rd.
    \ee
    \begin{theorem}\label{th.b2}
For any  $V\in L^1(\rd) + L^\infty(\rd)$ and $q \in {\mathbb
Z}_+$, we have
    \bel{sof31}
\mathcalU_B^* P_q VP_q \mathcalU_B = p_q\otimes {\rm Op}_q^{aw} (V_B).
    \ee
\end{theorem}
    For the proof of Theorem \ref{th.b2} we need some well known estimates for Berezin-Toeplitz
    operators:
    \begin{lemma} \label{l21} {\rm \cite[Lemma 5.1]{r0}, \cite[Lemma
5.1]{fr}}  Let $V  \in L^\ell(\rd)$, $\ell \in [1,\infty)$. Then
$P_q V P_q \in S_\ell( L^2(\rd))$, and $\|P_q V P_q\|_{\ell}^\ell
\leq \frac{B}{2\pi} \|V\|^\ell_{L^{\ell}(\rd)}$, $q \in {\mathbb
Z}_+$. Moreover, if $V \in L^1(\rd)$, then
    \bel{gr18}
    \Tr P_q V P_q = \frac{B}{2\pi} \int_{\rd} V(\bx) d\bx, \quad q \in \ze_+.
    \ee
\end{lemma}
\begin{proof}[Proof of Theorem \ref{th.b2}]
Assume at first  $V \in C^{\infty}_0(\rd)$. Then, by
\eqref{sof21}, \eqref{sof32}, and \eqref{8d01},
    \bel{8d07}
    {\mathcal U}_B^* P_q V P_q {\mathcal
U}_B = (p_q \otimes I_y) {\rm Op}^{w}({\bf V}_B) (p_q \otimes
I_y). \ee   Let $u \in {\mathcal S}(\rd)$. Set
$$
u_q(y) : =
\int_{\rd} u(x,y)\varphi_q(x)dx.
$$
 Then we have
$$
\langle{\mathcal U}_B^* P_q V P_q {\mathcal U}_B u,
u\rangle_{L^2(\rd)} = \langle{\rm Op}^{w}({\bf V}_B)
(\varphi_q\otimes u_q), (\varphi_q\otimes u_q)\rangle_{L^2(\rd)} =
$$
$$
\frac{1}{(2\pi)^{2}}\int_{\re^6} V_B ((y_1+y_2)/2-\xi, \eta -
(x_1+x_2)/2 )\, e^{i[(x_1-x_2)\xi + (y_1-y_2)\eta]} \,\times
$$
$$
 \varphi_q(x_1) u_q(y_1) \;
\varphi_q(x_2) \overline{u_q(y_2)} \;dx_1 dx_2 \,dy_1 dy_2 \,d\xi
d\eta =
$$
$$
\frac{1}{(2\pi)^{2}}\int_{\re^5} V_B((y_1+y_2)/2-y', \eta -
\eta')\,e^{i(y_1-y_2)\eta}\, \times
$$
$$
\left(\int_{\re} \varphi_q(\eta'+v/2) \varphi_q(\eta' -
v/2)e^{ivy'}dv\right) u_q(y_1) \overline{u_q(y_2)} d\eta'd\eta
dy'dy_1dy_2 =
$$
$$
\frac{1}{2\pi}\int_{\re^5} \Psi_q(y',\eta') V_B ((y_1+y_2)/2-y',
\eta - \eta') e^{i(y_1-y_2)\eta}u_q(y_1)
\overline{u_q(y_2)}dy_1dy_2dy' d\eta d\eta'=
$$
\bel{9d01}
 \langle {\rm Op}^{w}(V_B * \Psi_q)u_q, u_q \rangle_{L^2(\re)} =
 \langle {\rm Op}_q^{aw}(V_B)u_q, u_q \rangle_{L^2(\re)} = \langle(p_q \otimes {\rm Op}_q^{aw}(V_B))u, u\rangle_{L^2(\rd)}.
\ee To obtain the first identity, we have utilized Corollary
\ref{f41}. To establish the second identity, we have used
\eqref{gr52}, \eqref{8d01}, and \eqref{sof30}. To get the third
identity, we have changed the variables $x_1 = \eta' + v/2$, $x_2
= \eta' - v/2$, $\xi = y'$. To obtain the fourth identity, we have
used \eqref{gr50} -- \eqref{gr51}
  with $\xi' = y'$, $x' = \eta'$, and $x =0$, $\xi = 0$, taking
  into account that $\Psi_q(\eta', -y') = \Psi_q(y',\eta')$. To deduce the fifth identity, we have  applied
  \eqref{gr52}, bearing in mind the symmetry of
the convolution $\Psi_q * V_B = V_B * \Psi_q$, and for  the sixth
identity, we have   applied \eqref{9d05} with $s = V_B$. Finally,
the last identity is obvious. Now, \eqref{9d01} entails
\eqref{sof31} in the case $V \in C_0^{\infty}(\rd)$. \\
Further, let  $V \in L^1(\rd)$, and pick a sequence
$\left\{V_m\right\}$ of functions $V_m \in C_0^{\infty}(\rd)$ such
that $V_m \rightarrow V$ in $L^1(\rd)$ as $m \to \infty$. Then by
Lemma \ref{l21} and the unitarity of ${\mathcal U}_B$, we have
$$
\lim_{m \to \infty}\| {\mathcal U}_B^* P_q V_m P_q
{\mathcal U}_B - {\mathcal U}_B^* P_q V P_q {\mathcal U}_B\|_1 = 0.
$$
Similarly, it follows from \eqref{8d04} with $\ell = 1$ and
\eqref{9d05} that
$$
\lim_{m \to \infty}\| p_q \otimes {\rm Op}_q^{aw}(V_{m,B}) -
p_q \otimes {\rm Op}_q^{aw}(V_B)\|_1 = 0.
$$
Hence, \eqref{sof31} is valid for $V \in L^1(\rd)$.\\
 Finally, let now
$V = V_1 + V_2$ with $V_1\in L^1(\rd)$ and $V_2\in
L^{\infty}(\rd)$. Denote by $\chi_R$ the characteristic function
of a disk of radius $R>0$ centered at the origin. Then $V_1 +
\chi_R V_2 \in L^1(\rd)$. Evidently,
$$
\wlim_{R \to \infty} {\mathcal U}_B^* P_q (V_1 + \chi_R
V_2) P_q {\mathcal U}_B = {\mathcal U}_B^* P_q V P_q {\mathcal
U}_B,
$$
while  \eqref{8d03} entails
$$
\wlim_{R \to \infty} p_q \otimes {\rm Op}_q^{aw}((V_1 +
\chi_R V_2)_B) = p_q \otimes {\rm Op}_q^{aw}(V_B),
$$
which yields \eqref{sof31} in the general case.
\end{proof}
Combining Theorem~\ref{th.b2} and Corollary~\ref{f42}, we obtain
the following

\begin{follow}\label{fgr1}
Let  $V\in L^1(\rd) + L^\infty(\rd)$ and $q \in {\mathbb Z}_+$. Then  we have
    \bel{gr7}
\mathcalU_B^* P_q VP_q \mathcalU_B = p_q\otimes {\rm Op}^{w}
(V_B * \Psi_q).
    \ee
\end{follow}

\begin{remark} \label{r6}
 To the authors' best knowledge, the unitary equivalence
between the Toeplitz operators $P_q V P_q$, $q \in {\mathbb Z}_+$,
and $\Psi$DO with generalized  anti-Wick symbols in the context of
the spectral theory of perturbations of the Landau Hamiltonian,
was first shown in \cite{r0}. Related heuristic arguments  can be
found in \cite{rs, bhele}. In the case $q=0$ this equivalence is
closely related to the Segal-Bargmann transform in appropriate
holomorphic spaces which, in one form or another, plays an
important role in the semiclassical analysis performed in
\cite{tw, vb1, tvb, urvb}. Let us comment in more detail on this
relation. The Hilbert space $P_0 L^2(\rd)$ coincides with the
classical  Bargmann space
$$
\left\{f \in L^2(\rd) \, |\, f(\bx) = e^{-B|\bx|^2/4}g(\bx), \quad
\frac{\partial g}{\partial x} + i \frac{\partial g}{\partial y} =
0, \quad \bx = (x,y) \in \rd\right\}.
$$
Then the Segal-Bargmann transform $T_0: L^2(\re) \to P_0 L^2(\rd)$ is a unitary operator with integral kernel
$$
{\mathcal T}_0 : = \frac{1}{\sqrt{2}} \left(\frac{B}{\pi}\right)^{3/4} e^{-B((x+iy+2t)^2 - 2t^2 + |\bx|^2)/4},
\quad \bx = (x,y) \in \rd, \quad t \in \re,
$$
(see \cite[Lemma 3.1]{tp}). Fix $q \in {\mathbb Z}_+$. Denote by $M_q : L^2(\re) \to (p_q \otimes I_y) L^2(\rd)$
the unitary operator which maps $u \in L^2(\re)$ into $B^{1/4} \varphi_q(x) u(B^{1/2}y)$, $(x,y) \in \rd$,
and by ${\mathcal R} : L^2(\rd) \to L^2(\rd)$ the unitary operator generated by the rotation by angle $\pi/2$, i.e.
$({\mathcal R} u)(x,y) = u(y,-x)$, $(x,y) \in \rd$; note that $[{\mathcal R}, P_q] = 0$. Then we have
 $$
 T_0 = {\mathcal R} {\mathcal U}_B M_0.
 $$
{}From this point of view the operators $T_q : = {\mathcal R}
{\mathcal U}_B M_q$, $q \in {\mathbb N}$, could be called
generalized Segal-Bargmann transforms.
\end{remark}

\section{Analysis of $\Op (V_B * \Psi_q)$ and proof of Theorem 1.6}\label{sec.c}

\subsection{Reduction of $\Op(V_B * \Psi_q)$ to $\Op(V_B * \delta_{\sqrt{2q+1}})$}
In the sequel we will use the following notations.
For $k>0$, let $\delta_k$ be the $\delta$-function in $\rd$
supported on the circle of radius $k$ centered at the origin. More
precisely, the distribution $\delta_k \in {\mathcal S}'(\re^2)$ is defined by
$$
\delta_k(\varphi) :  =
\frac1{2\pi}\int_0^{2\pi}\varphi(k\cos\theta,k\sin\theta)d\theta,
\quad \varphi\in {\mathcal S}(\rd).
$$
Next, we denote by $\hat{f}$ the  Fourier transform of the
distribution $f \in {\mathcal S}'(\re^d)$,  unitary in $L^2(\re^d)$, i.e.
    \bel{end2}
\hat{f}(\xi) : = (2\pi)^{-d/2} \int_{\re^d} e^{-ix\cdot\xi} f(x)
dx, \quad \xi \in \re^d,
    \ee
for $f \in {\mathcal S}(\re^d)$.

\begin{lemma}\label{th.b3}
Let $V\in C_0^\infty(\R^2)$ and $B_0>0$.
Then for some constant $C=C(B_0)$ one has
\begin{multline}
\sup_{q\geq0}\sup_{B\geq B_0}\lambda_q^{3/4} B^{-1}
\norm{\Op(V_B * \Psi_q)-\Op(V_B*\delta_{\sqrt{2q+1}})}
\\
\leq
C\int_{\re^2} (\abs{\zeta}^{5/2}+\abs{\zeta}^6)\abs{\widehat V(\zeta)}d\zeta,
\label{c13}
\end{multline}
\begin{multline}
\sup_{q\geq0}\sup_{B\geq B_0}\lambda_q^{3/4} B^{-1}
\norm{\Op(V_B * \Psi_q)-\Op(V_B*\delta_{\sqrt{2q+1}})}_2
\\
\leq
C\left(\int_{\re^2} (\abs{\zeta}^{5}+\abs{\zeta}^{12})\abs{\widehat V(\zeta)}^2 d\zeta\right)^{1/2}.
\label{c14}
\end{multline}
\end{lemma}

The intuition behind this lemma is the convergence of $ \Psi_q$ to
$\delta_{\sqrt{2q+1}}$ in an appropriate sense as $q\to\infty$.
We also note the
similarity between the definition of $V_B * \delta_k$ and the
``classical'' formula \eqref{a1}. For brevity, we introduce the
short-hand notations
 \bel{end1}
s_q : = V_B * \Psi_q, \quad q \in {\mathbb Z}_+, \quad t_k : =
V_B*\delta_k, \quad k \in (0,\infty).
    \ee
\begin{proof}
First we represent the symbols $s_q$, $t_k$ in a form
convenient for our purposes. For $s_q$ we have
\bel{gr17}
s_q(z)= \int_{\R^2} e^{iz\zeta} \widehat
\Psi_q(\zeta)\widehat V_B(\zeta)d\zeta, \quad z\in\R^2.
    \ee
In the Appendix we will prove the formula
\begin{equation}
\widehat{\Psi_q}(\zeta) = (-1)^q \Psi_q(2^{-1} \zeta)/2, \quad q
\in {\mathbb Z}_+, \quad \zeta \in \R^2. \label{prvb30}
\end{equation}
By \eqref{gr17}, \eqref{prvb30}, and the definition \eqref{17a}
of $\Psi_q$,
$$
s_q(z)=\frac1{2\pi}\int_{\R^2} e^{iz\zeta}
\Lag_q(\abs{\zeta}^2/2)e^{-\abs{\zeta}^2/4}
\widehat V_B(\zeta) d\zeta.
$$
For $t_k$ we can write
\begin{equation}
t_k(z)=\frac{1}{2\pi}\int_{\rd}e^{iz\zeta}J_0(k\abs{\zeta})\widehat
V_B(\zeta)d\zeta, \quad z\in\rd,
\label{b14}
\end{equation}
since the integral representation
for the Bessel function $J_0$ can be written as
\begin{equation}
J_0(k |\zeta|) = 2\pi \hat{\delta}_k(\zeta), \quad \zeta\in\R^2.
\label{b13}
\end{equation}
Thus \eqref{c13}  (resp. \eqref{c14}) reduces to estimating the operator norm
(resp. the Hilbert-Schmidt norm) of the operator with the Weyl symbol
\begin{equation}
s_q(z)-t_{\sqrt{2q+1}}(z)
=
\frac{1}{2\pi}\int_{\R^2} e^{iz\zeta}
\bigl(\Lag_q(\abs{\zeta}^2/2)e^{-\abs{\zeta}^2/4}
-
J_0(\sqrt{2q+1}\abs{\zeta})\bigr)
\widehat V_B(\zeta)d\zeta, \quad q \in {\mathbb Z}_+.
\label{b5}
\end{equation}

In what follows the estimate
\begin{equation}
\Abs{{\rm L}_q(x) e^{-x/2} - J_0(\sqrt{(4q+2)x})} \leq
C(q^{-3/4}x^{5/4}+q^{-1}x^3), \qquad q\in\mathbb N, \quad x>0,
\label{b7}
\end{equation}
plays a key role.
This estimate is probably well known to experts,
but since we could not find it explicitly in the literature, we include its
proof  in the Appendix.

Let us prove the estimate \eqref{c13}.
Using the estimates \eqref{end3} and \eqref{b7}, we obtain:
\begin{multline}
\norm{\Op(V_B * \Psi_q)-\Op(V_B*\delta_{\sqrt{2q+1}})}
\leq
(2\pi)^{-1}
\norm{\hat{s_q}-\hat t_{\sqrt{2q+1}}}_{L^1(\re^2)}
\\
=
(2\pi)^{-1}
\int_{\R^2}\abs{\Lag_q(\abs{\zeta}^2/2)e^{-\abs{\zeta}^2/4}-J_0(\sqrt{2q+1}\abs{\zeta})}
\abs{\widehat V_B(\zeta)}d\zeta
\\
\leq
C
\int_{\R^2}(q^{-3/4}\abs{\zeta}^{5/2}+q^{-1}\abs{\zeta}^6)
\abs{\widehat V_B(\zeta)}d\zeta.
\label{b10}
\end{multline}
Recalling the definition \eqref{sof30} of $V_B$, we obtain
$\widehat V_B(\zeta)=B\widehat V_1(B^{1/2}\zeta)$, and so the l.h.s. of
\eqref{b10} can be estimated by
\begin{multline*}
CB q^{-3/4} \int_{\R^2} \abs{\zeta}^{5/2} \abs{\widehat V_1(B^{1/2}\zeta)}d\zeta
+
CB q^{-1} \int_{\R^2} \abs{\zeta}^{6} \abs{\widehat V_1(B^{1/2}\zeta)}d\zeta
\\
=
CB^{-5/4} q^{-3/4} \int_{\R^2} \abs{\zeta}^{5/2} \abs{\widehat V_1(\zeta)}d\zeta
+
CB^{-3} q^{-1} \int_{\R^2} \abs{\zeta}^{6} \abs{\widehat V_1(\zeta)}d\zeta.
\end{multline*}
This yields \eqref{c13}.

Next, let us prove the estimate \eqref{c14}.
By \eqref{prvb2} and the unitarity of the Fourier transform,
\begin{multline*}
\norm{\Op(s_q-t_k)}^2_2
=
(2\pi)^{-1}\int_{\R^2}\abs{\hat{s}_q(\zeta)-\hat{t}_{\sqrt{2q+1}}(\zeta)}^2 d\zeta
\\
=
(2\pi)^{-1}
\int_{\R^2}
\abs{\Lag_q(\abs{\zeta}^2/2)e^{-\abs{\zeta}^2/4}-J_0(\sqrt{2q+1}\abs{\zeta})}^2
\abs{\widehat V_B(\zeta)}^2 d\zeta.
\end{multline*}
Now using the estimate \eqref{b7} again, we obtain \eqref{c14} in a similar way
to the previous step of the proof.
\end{proof}

\subsection{Norm estimate of $\Op(V_B * \delta_k)$}
\begin{lemma}\label{lma.b3}
Let $V({\bf x})=\jap{{\bf x}}^{-\rho}$, $\rho>1$, and $B_0>0$.
Then
$$
\sup_{k>0} \sup_{B>B_0}
k B^{-1/2}
\norm{\Op(V_B * \delta_k)}<\infty.
$$
\end{lemma}
\begin{proof}
By Proposition~\ref{prvbp1}, it suffices to prove that for any
 differential operator $L$ with constant coefficients
we have
\bel{prvb14}
\sup_{k>0} \sup_{B>B_0}
k B^{-1/2}\sup_{z\in \rd} \abs{(LV_B* \delta_k)(z)} < \infty.
\ee
Note that, by the standard symbol
properties of $\langle \bx\rangle^{-\rho}$, we have
$$
|L V_B(\bx)|\leq CV_B(\bx), \quad \bx\in\rd,
$$
where $C$ depends only on $B_0$ and $L$.
Thus, it remains to prove that
\begin{equation}
\sup_{k>0} \sup_{B>B_0}
kB^{-1/2}
\sup_{z\in\rd}\abs{(V_B*\delta_k)(z)}<\infty.
\label{c15}
\end{equation}
We have
$$
V_B(z)=(B^{-1}\abs{z}^2+1)^{-\rho/2}.
$$
Take $z=(r,0)$, $r\geq0$. Then
\begin{multline*}
\left(V_B * \delta_k \right)(z) =
\frac1{2\pi}
\int_0^{2\pi}(B^{-1}(k\cos\theta-r)^2+B^{-1}(k\sin\theta)^2+1)^{-\rho/2}d\theta
\\
\leq \frac1{2\pi}
\int_0^{2\pi}(B^{-1}k^2(\sin\theta)^2+1)^{-\rho/2}d\theta
=
\frac2{\pi}
\int_0^{\pi/2}(B^{-1}k^2(\sin\theta)^2+1)^{-\rho/2}d\theta
\\
\leq \frac2{\pi}
\int_0^{\pi/2}(B^{-1}k^2(2\theta/\pi)^2+1)^{-\rho/2}d\theta
\leq
\frac2{\pi} \int_0^{\infty}(B^{-1}k^2(2\theta/\pi)^2+1)^{-\rho/2}d\theta
\\
= \frac{2B^{1/2}}{\pi k}
\int_0^{\infty}((2\theta/\pi)^2+1)^{-\rho/2}d\theta = CB^{1/2}/k.
\end{multline*}
This yields
$$
\sup_{z\in\rd}\abs{(V_B*\delta_k)(z)}\leq CB^{1/2}k^{-1},
$$
and \eqref{c15} follows.
\end{proof}

\subsection{Asymptotics of traces}
\begin{theorem}\label{th.b4}
Let $V\in C_0^\infty(\R^2)$. Then for each $\ell \in {\mathbb N}$, $\ell \geq 2$,
we have
    \bel{sof12}
\lim_{q \to \infty}\lambda_q^{(\ell-1)/2}
\Tr \bigl(\Op(t_{\sqrt{2q+1}})\bigr)^\ell
=
\frac{B^\ell}{2\pi}
\int_{{\mathbb T}} \, \int_{\R} \;\widetilde{V}(\omega,b)^\ell \,
db\, d\omega.
    \ee
\end{theorem}

The proof is based on the following  technical lemma.
\begin{lemma}\label{lma.b2}
Let $\ell\in\N$, $\ell\geq 2$, $f\in\mathcal S(\R^{2(\ell-1)})$,
and let the function $\varphi:\mathbb T^{\ell}\times \R^{2(\ell-1)}\to\R$
be given by
\begin{equation}
\varphi(\bomega,\bz) = \sum_{j=1}^{\ell-1} z_j \cdot
(\omega_{j+1}-\omega_j), \label{c1}
\end{equation}
where
$\bz=(z_1,\dots,z_{\ell-1})\in \R^{2(\ell-1)}$,
$\bomega=(\omega_1,\dots,\omega_\ell)\in\mathbb T^{\ell}\subset\R^{2\ell}$,
and $\cdot$ denotes the scalar product in $\rd$.
Then
\begin{multline}
\lim_{k\to\infty}k^{\ell-1}
\int_{\R^{2(\ell-1)}} \, \int_{\mathbb T^{\ell}} f(\bz)
e^{ik\varphi(\bomega,\bz)}\, d\bomega\, d\bz
\\
=
(2\pi)^{\ell-1}
\int_{\mathbb T} \, \int_{\R^{\ell-1}}
\, f(\alpha_1\omega,\alpha_2\omega,\dots, \alpha_{\ell-1}\omega)\, d\alpha_1 d\alpha_2 \cdots d\alpha_{\ell-1}\,d\omega .
\label{c2}
\end{multline}
\end{lemma}
\begin{proof}
The proof consists in  an application of the stationary phase
method.

We use the following parametrisation of the variables
$\bomega$, $\bz$:
\begin{align*}
\omega_\ell&=(\cos\theta,\sin\theta), \quad \theta\in[-\pi,\pi);
\\
\omega_j&=(\cos(\theta+\theta_j), \sin(\theta+\theta_j)),
\quad   \theta_j\in[-\pi,\pi), \quad j=1,\dots, \ell-1;
\\
z_j&=\alpha_j\omega_\ell+\beta_j\omega_\ell^\perp,
\quad \alpha_j, \beta_j\in\R, \quad  j=1,\dots, \ell-1.
\end{align*}
We write
$\balpha = (\alpha_1, \ldots, \alpha_{\ell - 1}) \in \re^{\ell-1}$,
$\bbeta = (\beta_1, \ldots, \beta_{\ell - 1}) \in\re^{\ell - 1}$,
$\btheta = (\theta_1,\ldots,\theta_{\ell-1}) \in[-\pi,\pi)^{\ell-1}$.
Using this notation, we can rewrite the integral in the l.h.s. of
\eqref{c2} as
\begin{multline}
\int_{\R^{2(\ell-1)}} \, \int_{\mathbb T^{\ell}} f(\bz)
e^{ik\varphi(\bomega,\bz)}\,d\bomega \,d\bz
\\
=
\int_{-\pi}^\pi \, \int_{(-\pi,\pi)^{\ell-1}}\,
\int_{\R^{\ell-1}}\, \int_{\R^{\ell-1}} \,
F(\balpha, \bbeta,\btheta,\theta)e^{ik \Phi(\balpha,\bbeta,\btheta)} \,d\bbeta\,d\balpha \,d\btheta\,d\theta,
\label{c3}
\end{multline}
where
$$
F(\balpha, \bbeta,\btheta,\theta)
=
f(\alpha_1\omega_\ell+\beta_1\omega_\ell^\perp,\dots,\alpha_{\ell-1}\omega_\ell+\beta_{\ell-1}\omega_\ell^\perp),
$$
and
$$
\Phi(\balpha,\bbeta,\btheta)
=
\alpha_1(1-\cos\theta_1)-\beta_1\sin\theta_1,
\text{ if $\ell=2$,}
$$
\begin{multline*}
\Phi(\balpha,\bbeta,\btheta)
=
\alpha_{\ell-1}-(\alpha_1\cos\theta_1+\beta_1\sin\theta_1)
\\
+
\sum_{j=2}^{\ell-1}((\alpha_{j-1}-\alpha_j)\cos\theta_j+(\beta_{j-1}-\beta_j)\sin\theta_j),
\text{ if $\ell\geq 3$.}
\end{multline*}

Let us consider the stationary points of the phase function
$\Phi$. By a direct calculation,
$\nabla\Phi(\balpha,\bbeta,\btheta)=0$ if and only if $\bbeta=0$
and $\btheta=0$. By a standard localisation argument, it follows
that the asymptotics of the integral \eqref{c3} will not change if
we multiply  $F$ by a function $\chi=\chi(\bbeta,\btheta)$, $\chi
\in C^\infty(\R^{\ell-1}\times[-\pi,\pi)^{\ell-1})$, such that
$\chi(\bbeta,\btheta)=1$ in an open neighbourhood of the origin
$\bbeta=0$, $\btheta=0$, and $\chi(\bbeta,\btheta)=0$ if
$\abs{\bbeta}\geq1/2$ or $\abs{\btheta}\geq\pi/2$.

Let us write
\begin{multline}
\int_{-\pi}^\pi \, \int_{(-\pi,\pi)^{\ell-1}} \,
\int_{\R^{\ell-1}} \, \int_{\R^{\ell-1}} \,
F(\balpha, \bbeta,\btheta,\theta)\chi(\bbeta,\btheta)
e^{ik \Phi(\balpha,\bbeta,\btheta)} \, d\bbeta \, d\balpha \, d\btheta \, d\theta
\\
=
\int_{-\pi}^\pi
\int_{\R^{\ell-1}}
I(k;\balpha,\theta) \, d\balpha \, d\theta ,
\label{c4}
\end{multline}
where
$$
I(k;\balpha,\theta)
=
\int_{(-\pi,\pi)^{\ell-1}} \,
\int_{\R^{\ell-1}} \,
F(\balpha, \bbeta,\btheta,\theta)\chi(\bbeta,\btheta)
e^{ik \Phi(\balpha,\bbeta,\btheta)} \,d\bbeta  \,d\btheta.
$$
Let us fix $\balpha$, $\theta$ and compute the
asymptotics of the integral
$I(k;\balpha,\theta)$ as $k\to\infty$.
A direct calculation shows that the stationary phase
equations
$$
\frac{\partial\Phi}{\partial\beta_j}=0,
\quad
\frac{\partial\Phi}{\partial\theta_j}=0,
\quad
j=1,\dots,\ell-1
$$
are simultaneously satisfied on the support of $\chi$ if and only
if $\bbeta=0$, $\btheta=0$. In order to apply the stationary phase
method, we need to compute the determinant and the signature (i.e.
the difference between the number of positive and negative
eigenvalues) of the Hessian of $\Phi(\balpha,\bbeta,\btheta)$ with
respect to the variables $\bbeta$, $\btheta$. Let us denote this
Hessian by $H(\balpha)$. The $2(\ell - 1) \times 2(\ell - 1)$
matrix $H(\balpha)$ can be represented in a block form
$$
H(\balpha) = \left(
\begin{array} {cc}
H_{11}(\balpha) & H_{12}(\balpha)\\
H_{21}(\balpha) & H_{22}(\balpha)
\end{array}
\right)
$$
 where
$$
H_{11}(\balpha) : = \left\{\frac{\partial^2 \Phi}{\partial \beta_p
\partial \beta_q}(\balpha,0,0) \right\}_{p,q=1}^{\ell-1}, \quad
H_{12}(\balpha) : = \left\{\frac{\partial^2 \Phi}{\partial \beta_p
\partial \theta_q}(\balpha,0,0) \right\}_{p,q=1}^{\ell-1},
$$
$$
H_{21}(\balpha) : = \left\{\frac{\partial^2 \Phi}{\partial
\theta_p
\partial \beta_q}(\balpha,0,0) \right\}_{p,q=1}^{\ell-1}, \quad
H_{22}(\balpha) : = \left\{\frac{\partial^2 \Phi}{\partial
\theta_p
\partial \theta_q}(\balpha,0,0) \right\}_{p,q=1}^{\ell-1}.
$$
An explicit computation shows that
$$
H_{11} = 0, \quad H_{12} = \left\{-\delta_{q,p} +
\delta_{q,p+1}\right\}_{p,q=1}^{\ell - 1},
$$
$$
H_{21} = H_{12}^T, \quad H_{22}(\balpha) = {\rm
diag}\left\{\alpha_1, \alpha_2 - \alpha_1,\ldots, \alpha_{\ell -
1} - \alpha_{\ell - 2}\right\}.
$$
Hence,
$$
{\rm det}_{2(\ell - 1)} H(\balpha) = {\rm det}_{\ell - 1}(- H_{12}
H_{21}) = (-1)^{\ell -1};
$$
in particular, our stationary point is non-degenerate.

In order to calculate the signature
of $H(\balpha)$, note that since $\det H(\balpha)\not=0$
for all $\balpha$ and $H(\balpha)$ depends smoothly (in fact,
polynomially) on $\balpha$, the signature is independent of $\balpha$.
Some elementary analysis shows that $\sign H(0)=0$.

Now we can apply a suitable version of the stationary phase
method (see e.g. \cite[Chapter 1]{guist} or \cite[Chapter III,
Section 2]{fed}) to calculate the asymptotics of
$I(k;\balpha,\theta)$. This yields
$$
\lim_{k\to\infty}k^{\ell-1}I(k;\balpha,\theta)
=
(2\pi)^{\ell-1}F(\balpha,0,0,\theta).
$$
Using, for example, the Lebesgue dominated convergence
theorem, one concludes that the above asymptotics can
be integrated over $\balpha$ and $\theta$,
see \eqref{c4}. This yields the required result \eqref{c2}.
\end{proof}

\begin{proof}[Proof of Theorem~\ref{th.b4}]
We use the notation $t_k$, see \eqref{end1}.
Denote by $t_{k, \ell}$ the Weyl symbol of the operator
$(\Op(t_k))^\ell$,  i.e.
$$
\Op(t_{k, \ell}) =(\Op(t_{k}))^\ell, \quad \ell=2,3,\dots.
$$
By the standard Weyl
pseudodifferential calculus
(see e.g. \cite[Chapter 7, Eq. (14.21)]{tayii}), for $\zeta_\ell \in \rd$ we have
$$
\hat{t}_{k, \ell}(\zeta_\ell)
=
(2\pi)^{-\ell + 1}
\int_{\re^{2(\ell-1)}} \hat{t}_k(\zeta_\ell - \zeta_{\ell - 1})
\hat{t}_k(\zeta_{\ell - 1} - \zeta_{\ell - 2}) \cdots
\hat{t}_k(\zeta_1) e^{\frac{i}{2}\sum_{j=2}^{\ell} \sigma(\zeta_j,
\zeta_{j-1})} d\bzeta
$$
where $\sigma(\cdot, \cdot)$ is the symplectic form in
$\rd \times\rd$,
and $\bzeta=(\zeta_1,\dots,\zeta_{\ell-1})$.
It follows that
\begin{multline}
\Tr \bigl(\Op(t_k)\bigr)^\ell
=
\Tr \Op(t_{k,\ell}) =
\frac1{2\pi}\int_{\R^2}t_{k,\ell}(x)dx = \hat t_{k,\ell}(0)
\\
=
(2\pi)^{-\ell + 1}
\int_{\re^{2(\ell-1)}} \hat{t}_k(-\zeta_{\ell - 1})
\hat{t}_k(\zeta_{\ell - 1} - \zeta_{\ell - 2}) \cdots \hat{t}_k(\zeta_1)
\exp(\frac{i}{2}\sum_{j=2}^{\ell-1} \sigma(\zeta_j,\zeta_{j-1}))
d\bzeta,
\label{c5}
\end{multline}
where we use the convention that $\sum_{j=2}^{\ell-1}=0$ if
$\ell=2$.

Recalling \eqref{b13}, \eqref{b14}, we get
$$
\hat t_k(\zeta)
=
\frac1{2\pi}\widehat V_B(\zeta)\int_{\mathbb T}e^{-ik\omega\zeta}d\omega,
\quad
\zeta\in\R^2,
$$
and so, substituting into \eqref{c5}, we get
$$
\Tr\bigl(\Op(t_k)\bigr)^\ell =
\int_{\R^{2(\ell-1)}}\,\int_{\mathbb T^\ell}\, f(\bzeta)
e^{ik\varphi(\bomega,\bzeta)}\, d\bomega \, d\bzeta,
$$
where
$$
f(\bzeta)
=
(2\pi)^{-2\ell+1}
\widehat V_B(-\zeta_{\ell - 1})
\widehat V_B(\zeta_{\ell - 1} - \zeta_{\ell - 2}) \cdots \widehat V_B(\zeta_1)
\exp\bigl({\textstyle \frac{i}{2} \sum_{j=2}^{\ell-1} \sigma(\zeta_j,\zeta_{j-1})}\bigr),
$$
and $\varphi$ is given by \eqref{c1}.
Applying Lemma~\ref{lma.b2}, we obtain
\begin{multline}
\lim_{k\to\infty} k^{\ell-1}
\Tr\bigl(\Op(t_k)\bigr)^\ell
\\
= (2\pi)^{-\ell}\int_{\mathbb T} \int_{\R^{\ell-1}}\,
\widehat V_B(-\alpha_{\ell-1}\omega) \widehat
V_B((\alpha_{\ell-1}-\alpha_{\ell-2})\omega) \cdots\widehat
V_B((\alpha_2-\alpha_1)\omega)\widehat V_B(\alpha_1\omega) \,
d\balpha \, d\omega. \label{c6}
\end{multline}

It remains to transform the last identity into \eqref{sof12}.
We have
$$
\widehat V_B(\alpha \omega)
=
B\int_{\R} e^{-i\alpha
bB^{1/2}}\widetilde{V_1}(\omega,b)\,db
$$
where, in accordance with  \eqref{sof30}, we use the notation
$V_1(x,y) = V(-y,-x)$, $(x,y) \in \rd$. Therefore,
\begin{multline}
(2\pi)^{-\ell}\int_{\mathbb T} \int_{\R^{\ell-1}}\, \widehat
V_B(-\alpha_{\ell-1}\omega) \widehat
V_B((\alpha_{\ell-1}-\alpha_{\ell-2})\omega) \cdots\widehat
V_B((\alpha_2-\alpha_1)\omega)\widehat V_B(\alpha_1\omega) \,
d\balpha \, d\omega
\\
= \frac{B^{(1+\ell)/2}}{2\pi} \int_{\mathbb T} \, \int_\R
\widetilde{V_1}(\omega,b)^\ell \, db \, d\omega
=
\frac{B^{(1+\ell)/2}}{2\pi} \int_{\mathbb T} \, \int_\R
\widetilde{V}(\omega,b)^\ell \, db \, d\omega. \label{end10}
\end{multline}
Now \eqref{sof12} follows from \eqref{c6} and \eqref{end10}.
\end{proof}

\subsection{Proof of Theorem 1.6.}
(i) First we observe that
\begin{multline*}
\norm{P_q V P_q}
\leq
\norm{P_q\jap{\cdot}^{-\rho/2}}
\norm{\jap{\cdot}^\rho V}
\norm{\jap{\cdot}^{-\rho/2}P_q}
\\
\leq
\norm{V}_{X_\rho}
\norm{P_q\jap{\cdot}^{-\rho/2}}
\norm{\jap{\cdot}^{-\rho/2}P_q}
=
\norm{V}_{X_\rho}
\norm{P_q\jap{\cdot}^{-\rho}P_q},
\end{multline*}
and so it suffices to consider the case $V(\bx)=\jap{\bx}^{-\rho}$.

Next, by Corollary~\ref{fgr1}, we have
$$
\norm{P_q VP_q}
=
\norm{\Op (V_B* \Psi_q)}.
$$
In order to estimate the norm of $\Op (V_B* \Psi_q)$, we  use
Lemmas~\ref{th.b3} and \ref{lma.b3}.
We note that for $V(\bx)=\jap{\bx}^{-\rho}$, $\rho>1$, we have
$\widehat V\in L^1$ and
$$
\abs{\widehat V(\zeta)}\leq C_N \abs{\zeta}^{-N},
\quad
\abs{\zeta}\geq 1,
$$
for all $N\geq 1$ (see e.g. \cite[Chapter XII, Lemma 3.1]{tay}).
Thus, the integral in the r.h.s. of \eqref{c13} is convergent, and so the proof
of  \eqref{c13} applies to $V(\bx)=\jap{\bx}^{-\rho}$, $\rho>1$.
Now, combining Lemmas~\ref{th.b3} and \ref{lma.b3}, we get
\begin{multline*}
\sup_{q\geq0}\sup_{B\geq B_0}
\lambda_q^{1/2} B^{-1} \norm{\Op(V_B* \Psi_q)}
\\
\leq
\sup_{q\geq0}\sup_{B\geq B_0}
\lambda_q^{3/4} B^{-1}
\norm{\Op(V_B* \Psi_q)-\Op(V_B*\delta_{\sqrt{2q+1}})}
\\
+
\sup_{q\geq0}\sup_{B\geq B_0}
\lambda_q^{1/2} B^{-1}
\norm{\Op(V_B*\delta_{\sqrt{2q+1}})}
<\infty,
\end{multline*}
which proves the required estimate.

(ii)
 As in the proof of part (i), we may assume $V(\bx)=\jap{\bx}^{-\rho}$.
First let us consider the case $\rho>2$, $\ell=1$.
By Lemma~\ref{l21} with $\ell=1$ we have
$$
B^{-1}\norm{P_qVP_q}_1
=
\frac1{2\pi}
\int_{\R^2}V(\bx)d\bx
\leq
\frac1{2\pi}
\norm{V}_{X_\rho}
\int_{\R^2}\jap{\bx}^{-\rho}d\bx,
$$
which proves \eqref{b1} in this case.

Let us consider the case of a general $\ell$.
For a fixed $s>1$ and any
$\ell\in[1,\infty]$, let
$$
M_q^{(\ell)} = B^{-1}\lambda_q^{\frac12-\frac1{2\ell}}
P_q \jap{\cdot}^{-s (1+\frac1\ell)} P_q;
$$
for $\ell=\infty$, one should replace $1/\ell$ by $0$.
By the previous step of the proof and part (i) of the theorem,
$$
\sup_{q\geq0}\sup_{B\geq B_0}
\norm{M_q^{(1)}}_1
\leq
C_1<\infty, \quad
\sup_{q\geq0}\sup_{B\geq B_0}
\norm{M_q^{(\infty)}}\leq C_\infty<\infty,
$$
where the constants $C_1$, $C_\infty$ depend only on $B_0$ and $s$.
Applying the Calderon-Lions interpolation theorem
(see e.g. \cite[Theorem IX.20]{RS2}), we get
$$
\sup_{q\geq0}\sup_{B\geq B_0}
\norm{M_q^{(\ell)}}_\ell
\leq
C_1^{1/\ell}C_\infty^{(\ell-1)/\ell}
<
\infty
$$
for all $\ell \geq 1$. It is easy to see that the last statement is
equivalent to \eqref{b1}.

(iii)
First let us note that the case $\ell=1$ is straightforward.
Indeed, if the integer $\ell =1$ is
admissible, i.e. if $\ell = 1 > 1/(\rho-1)$, then $\rho> 2$, $
V \in L^1(\rd)$, and \eqref{gr18} yields the identity
$$
\Tr P_q V P_q
=
\frac{B}{2\pi} \int_{\rd} V(\bx) d\bx
=
\frac{B}{2\pi} \int_{\mathbb T} \int_\re \widetilde{V}(\omega,b)
\, db \, d\omega.
$$
Thus, we may now assume $\ell\geq2$.
We will first prove the required identity \eqref{12} for
$V\in C_0^\infty(\R^2)$ and then use a limiting argument to
extend it to all $V\in X_\rho$.
Denote
\begin{equation}
\gamma_\ell(V)
=
\frac{B^\ell}{2\pi} \int_{{\mathbb T}} \, \int_{\R} \;
\widetilde{V}(\omega,b)^\ell \, db \, d\omega.
\label{c11}
\end{equation}
By Corollary~\ref{fgr1}, we have
$$
\Tr(P_qVP_q)^\ell =
\Tr\bigl(\Op(V_B*\Psi_q)\bigr)^\ell,
$$
and Theorem~\ref{th.b4} says
$$
\lim_{q \to \infty}
\lambda_q^{(\ell-1)/2}
\Tr\bigl(\Op(V_B*\delta_{\sqrt{2q+1}})\bigr)^\ell
= \gamma_\ell(V).
$$
Thus, it suffices to prove that
\begin{equation}
\lim_{q \to \infty}
\lambda_q^{(\ell-1)/2}
\abs{\Tr\bigl(\Op(V_B*\Psi_q)\bigr)^\ell
-
\Tr\bigl(\Op(V_B*\delta_{\sqrt{2q+1}})\bigr)^\ell}
=
0.
\label{c16}
\end{equation}
In order to prove \eqref{c16},
let us first display an elementary estimate
\begin{equation}
\abs{\Tr \left(A_1^\ell\right) - \Tr\left(A_2^\ell\right)} \leq
\ell \max\{\norm{A_1}^{\ell-1}_\ell,
\norm{A_2}^{\ell-1}_\ell\}\norm{A_1-A_2}_\ell; \label{c10}
\end{equation}
here $\ell\in\N$ and $A_n\in S_\ell$, $n=1,2$. The estimate
follows from the formula
$$
A_1^\ell-A_2^\ell = \sum_{j=0}^{\ell-1}
A_1^{\ell-j-1}(A_1-A_2)A_2^j
$$
and the H\"older type inequality for the $S_\ell$ classes.

Next,
using Corollary~\ref{fgr1} and part (ii) of the Theorem, we get
\begin{equation}
\limsup_{q\to\infty}
\lambda_q^{(\ell-1)/(2\ell)}\norm{\Op(V_B*\Psi_q)}_\ell
=
\limsup_{q\to\infty}
\lambda_q^{(\ell-1)/(2\ell)}\norm{P_q V P_q}_\ell
<
\infty.
\label{c17}
\end{equation}

Further, by estimate \eqref{c14}, using the assumption $\ell\geq2$, we get
\begin{multline}
\limsup_{q\to\infty}
\lambda_q^{(\ell-1)/(2\ell)}
\norm{\Op(V_B*\Psi_q)-\Op(V_B*\delta_{\sqrt{2q+1}})}_\ell
\\
\leq
\limsup_{q\to\infty}
\lambda_q^{1/2}
\norm{\Op(V_B*\Psi_q)-\Op(V_B*\delta_{\sqrt{2q+1}})}_2
=0.
\label{c18}
\end{multline}
Combining \eqref{c10}, \eqref{c17} and \eqref{c18},
we obtain \eqref{c16} for $V\in C_0^\infty$; thus, \eqref{12} is proven for
this class of potentials.

It remains to extend \eqref{12} to all potentials
$V\in C(\re^2)$ that satisfy \eqref{rho}.
For $\ell>1/(\rho-1)$, denote
\begin{align*}
\Delta_\ell(V)
&=
\limsup_{q\to\infty}
\lambda_q^{(\ell-1)/2}\Tr(P_qVP_q)^\ell,
\\
\delta_\ell(V)
&=
\liminf_{q\to\infty}
\lambda_q^{(\ell-1)/2}\Tr(P_qVP_q)^\ell.
\end{align*}
Above we have proven that
\begin{equation}
\Delta_\ell(V)=\delta_\ell(V)=\gamma_\ell(V)
\label{c7}
\end{equation}
for all potentials $V\in C_0^\infty(\R^2)$; now we need to extend
this identity to all $V\in X_\rho$.
From \eqref{a4} we obtain, similarly to \eqref{c10},
\begin{multline*}
\abs{\gamma_\ell(V_1)-\gamma_\ell(V_2)}
\leq
\frac{B^\ell}{2\pi}
\int_{\mathbb T} \int_{\R}
\abs{\widetilde V_1(\omega,b)^\ell-\widetilde V_2(\omega,b)^\ell}db d\omega
\\
\leq
\frac{B^\ell}{2\pi}
C\max\{\norm{V_1}_{X_\rho}^{\ell-1},\norm{V_2}_{X_\rho}^{\ell-1}\}\norm{V_1-V_2}_{X_\rho}
\int_{\mathbb T} \int_{\R} \jap{b}^{(1-\rho)\ell}db\, d\omega.
\end{multline*}
It follows that $\gamma_\ell$ is a continuous functional on $X_\rho$.
Similarly, using \eqref{c10} and part (ii) of the Theorem, we get
$$
\limsup_{q\to\infty}
\lambda_q^{(\ell-1)/2}\abs{\Tr(P_qV_1P_q)^\ell-\Tr(P_qV_2P_q)^\ell}
\leq
C\max\{\norm{V_1}_{X_\rho}^{\ell-1}, \norm{V_2}_{X_\rho}^{\ell-1}\}
\norm{V_1-V_2}_{X_\rho},
$$
and so the functionals $\Delta_\ell$, $\delta_\ell$ are continuous on $X_\rho$.
It follows that \eqref{c7} extends by continuity from $C_0^\infty$ to
the closure $X_\rho^0$ of $C_0^\infty$ in $X_\rho$.
In order to prove \eqref{c7} for all $V\in X_\rho$, one can argue
as follows. For a given $\ell>1/(\rho-1)$, choose $\rho_1$ such
that $1<\rho_1<\rho$ and $\ell>1/(\rho_1-1)$. Then $X_\rho\subset
X_{\rho_1}^0$ and by the same argument as above, \eqref{c7} holds
true for all $V\in X_{\rho_1}^0$. \qed

\section{Proof of Proposition 1.1 and Theorem 1.3}\label{s3}

As already indicated, this section heavily uses the construction of \cite{korpu}.

\subsection{Proof of Proposition 1.1.}\label{s3a}

Set $R_0(z):=(H_0-zI)^{-1}$.
By the Birman-Schwinger principle, if $\lambda \in \re \setminus \cup_{q = 0}^{\infty}\{\lambda_q\}$ is an eigenvalue of the operator $H$, then $-1$ is an eigenvalue of the operator $|V|^{1/2} R_0(\lambda) V^{1/2}$. Hence, it
suffices to show that for some $C>0$ and all sufficiently large
$q$, we have
\begin{equation}
\norm{\abs{V}^{1/2}R_0(\lambda)\abs{V}^{1/2}}<1,
\quad\text{ for all }
\lambda\in[\lambda_q-B,\lambda_q+B],
\quad
\abs{\lambda-\lambda_q}>\frac{C}{\sqrt{q}}.
\label{d3}
\end{equation}
Choose $m\in\mathbb N$ sufficiently large so that
$\norm{V}/\lambda_m<1/2$,
and write $R_0(\lambda)$ as
$$
R_0(\lambda)=\sum_{k=q-m}^{q+m}\frac{P_k}{\lambda_k-\lambda}+\widetilde R_0(\lambda).
$$
Then, for $\lambda\in[\lambda_q-B,\lambda_q+B]$,
$$
\norm{\abs{V}^{1/2}R_0(\lambda)\abs{V}^{1/2}}
\leq
\sum_{k=q-m}^{q+m}
\frac{\norm{\abs{V}^{1/2}P_k\abs{V}^{1/2}}}{\abs{\lambda_k-\lambda}}
+
\norm{\abs{V}^{1/2}\widetilde R_0(\lambda)\abs{V}^{1/2}}.
$$
By the choice of $m$, one has
$$
\norm{\abs{V}^{1/2}\widetilde R_0(\lambda)\abs{V}^{1/2}}
\leq
\norm{\abs{V}^{1/2}}(1/\lambda_m)\norm{\abs{V}^{1/2}}
=
\norm{V}/\lambda_m
<1/2.
$$
On the other hand, by Theorem~\ref{th13}(i),
$$
\sum_{k=q-m}^{q+m}
\frac{\norm{\abs{V}^{1/2}P_k\abs{V}^{1/2}}}{\abs{\lambda_k-\lambda}}
\leq(2m+1)O(q^{-1/2})\max_{q-m\leq k\leq q+m}\abs{\lambda_k-\lambda}^{-1}
=O(q^{-1/2})\abs{\lambda_q-\lambda}^{-1}.
$$
Thus, we get \eqref{d3} for sufficiently large $C>0$.
\qed

\subsection{Resolvent estimates}\label{s3b}
Let $\Gamma_q$ be a positively oriented circle of center
$\lambda_q$ and  radius $B$.
\begin{lemma}\label{lma.d1}
Let $V$ satisfy \eqref{rho}. Then for any $\ell>1$,
$\ell>1/(\rho-1)$, one has
\begin{align}
\sup_{z\in\Gamma_q}
\norm{\abs{V}^{1/2}R_0(z)\abs{V}^{1/2}}_{\ell}
&=
O(q^{-(\ell-1)/2\ell}\log q),
\quad q\to\infty,
\label{d1}
\\
\sup_{z\in\Gamma_q}
\norm{\abs{V}^{1/2}R_0(z)}_{2\ell}
&=
O(q^{-(\ell-1)/4\ell}\log q),
\quad q\to\infty.
\label{d2}
\end{align}
\end{lemma}
\begin{proof}
Let us prove \eqref{d1}.
Using the estimate \eqref{b1}, we get
for $z\in\Gamma_q$:
\begin{multline*}
\norm{\abs{V}^{1/2}R_0(z)\abs{V}^{1/2}}_{\ell}
\leq
\sum_{k=0}^\infty
\frac{\norm{\abs{V}^{1/2}P_k\abs{V}^{1/2}}_{\ell}}{\abs{\lambda_k-z}}
\leq
\sum_{k=0}^\infty\frac{C(1+k)^{-(\ell-1)/2\ell}}{\abs{\lambda_k-z}}
\\
\leq
C\int_0^{q-1}\frac{(1+x)^{-(\ell-1)/2\ell}}{\abs{B(2x+1)-z}}dx
+C\int_{q+1}^\infty\frac{(1+x)^{-(\ell-1)/2\ell}}{\abs{B(2x+1)-z}}dx
+
O(q^{-(\ell-1)/2\ell})
\\
=O(q^{-(\ell-1)/2\ell}\log q),
\end{multline*}
as $q\to\infty$.
This proves \eqref{d1}.
Using the fact that
$$
\norm{\abs{V}^{1/2}P_q}_{2\ell}^2
=
\norm{\abs{V}^{1/2}P_q\abs{V}^{1/2}}_{\ell},
$$
one proves the estimate \eqref{d2} in the same way.
\end{proof}

\subsection{Proof of Lemma 1.5.}\label{s3c}
The fact that $(P_qVP_q)^\ell\in S_1$ follows directly from Theorem~\ref{th13}.
Let, as above, $\Gamma_q$ be a positively oriented circle with the
centre $\lambda_q$ and radius $B$. Let $q$ be sufficiently large
so that (see Proposition~\ref{p11})  the contour $\Gamma_q$ does
not intersect the spectrum of $H$. We will use the formula
\bel{32} (H-\lambda_q)^\ell \one_{(\lambda_q - B, \lambda_q +
B)}(H) = -\frac{1}{2\pi i} \int_{\Gamma_q} (z-\lambda_q)^\ell
R(z)dz, \ee where $R(z)= (H-zI)^{-1}$. Let us expand the resolvent
$R(z)$ in the r.h.s. of \eqref{32} in the standard perturbation
series:
\bel{37} R(z) = R_0(z) + \sum_{j=1}^{\infty} (-1)^j
R_0(z)(V R_0(z))^j. \ee Let us discuss the convergence of these
series for $z\in\Gamma_q$, $q$ large. Denote $W=\abs{V}^{1/2}$,
$W_0=\sign(V)$. For $j\geq\ell$, we have
\begin{multline}
\norm{R_0(z)(VR_0(z))^j}_{1}
=
\norm{(R_0(z)W)(W_0WR_0(z)W)^{j-1}W_0(WR_0(z))}_{1}
\\
\leq
\norm{R_0(z)W}_{2\ell}
\norm{WR_0(z)W}_{\ell}^{j-1}
\norm{WR_0(z)}_{2\ell},
\quad
j\geq\ell.
\label{d1a}
\end{multline}
Applying Lemma~\ref{lma.d1}, we get that the series
in the r.h.s. of \eqref{37} converges in the trace norm
for $z\in\Gamma_q$ and $q$ sufficiently large
(note that although the tail of the series converges in the trace
class, the series itself is not necessarily trace class).

Next, it is easy to see that the integrals
$$
\int_{\Gamma_q} (z-\lambda_q)^\ell R_0(z)(VR_0(z))^j dz
$$
with $j<\ell$ vanish (since the integrand is analytic inside
$\Gamma_q$). Thus, recalling \eqref{d1a}, we obtain that the
operator $(H-\lambda_q)^\ell \one_{(\lambda_q - B, \lambda_q +
B)}(H)$ belongs to the trace class and
\begin{multline}
\Tr\{(H- \lambda_q)^\ell \one_{(\lambda_q - B, \lambda_q +
B)}(H)\}
\\
=
\label{38}
-\frac{1}{2\pi i}  \sum_{j=\ell}^{\infty} (-1)^j \int_{\Gamma_q} (z -
\lambda_q)^\ell
\Tr [R_0(z) (VR_0(z))^j]dz.
\end{multline}
Integrating by parts in each term of this series and computing
the term with $j=\ell$ by
the residue theorem, we obtain
\begin{multline}
\Tr\{(H - \lambda_q)^\ell
\one_{(\lambda_q - B, \lambda_q +B)}(H)\}
\\
=
\label{310}
\Tr(P_qVP_q)^\ell + \frac{\ell}{2\pi i}
\sum_{j=\ell+1}^{\infty} \frac{(-1)^j}{j}
\int_{\Gamma_q} (z - \lambda_q)^{\ell-1}
\Tr (V R_0(z))^j dz.
\end{multline}
It remains to estimate the series in the r.h.s.
of \eqref{310}.
This can be easily done by using Lemma~\ref{lma.d1}.
Similarly to \eqref{d1a}, we have
$$
\abs{\Tr(VR_0(z))^j}
\leq
\norm{WR_0(z)W}^j_{j} \leq \norm{WR_0(z)W}^j_{\ell}, \quad j \geq \ell,
$$
and so
$$
\left\lvert
\int_{\Gamma_q} (z - \lambda_q)^{\ell-1}
\Tr (V R_0(z))^j dz
\right\rvert
\leq
C_1(C_2 q^{-(\ell-1)/2\ell}\log q)^j,
$$
for all sufficiently large $q$.
Thus, the series in the r.h.s. of \eqref{310}
can be estimated by
$$
C_1\sum_{j=\ell+1}^\infty C_2^j q^{-\frac{\ell-1}{2\ell}j}(\log q)^j.
$$
For all sufficiently large $q$,
this series converges and can be estimated
as $o(q^{-(\ell-1)/2})$.
\qed
\subsection{Proof of Theorem 1.3.}\label{s32}

Let $R \geq C_1$  where $C_1$ is the constant from
Proposition~\ref{p11}. Then
    \bel{312}
\one_{[-R,R]}(\lambda_q^{1/2}(H-\lambda_q)) = \one_{(\lambda_q -
B, \lambda_q + B)}(H), \quad  q \in {\mathbb Z}_+. \ee

Next, choose $R\geq C_1$ so large that $\supp \varrho \subset [-R,
R]$. Let $\ell_0$ be an even natural number satisfying
$\ell_0>1/(\rho-1)$. Since $\varrho (\lambda)$ by assumption
vanishes near $\lambda = 0$, the function
$\varrho(\lambda)/\lambda^{\ell_0}$ is smooth. Applying the
Weierstrass approximation theorem to this function on the interval
$[-R,R]$, we obtain that for any $\varepsilon>0$ there exist
polynomials $P_+$, $P_-$ such that
    \bel{313}
    P_{\pm}(0) = P'_{\pm}(0) =\dots=P^{(\ell_0-1)}_{\pm}(0)=0,
    \ee
    \bel{314}
    P_-(\lambda) \leq \varrho(\lambda) \leq P_+(\lambda), \quad \forall
    \lambda \in [-R, R],
    \ee
    \bel{315}
    P_+(\lambda) - P_-(\lambda) \leq \varepsilon \lambda^{\ell_0}, \quad \forall
    \lambda \in [-R, R].
    \ee
    Thus, we can write
    $$
    \one_{[-R,R]}(\lambda) P_-(\lambda) \leq \varrho(\lambda) \leq \one_{[-R,R]}(\lambda)
    P_+(\lambda),
    $$
for any $\lambda \in [-R, R]$, and therefore
    $$
    \Tr \{\one_{[-R,R]}(\lambda_q^{1/2}(H-\lambda_q))\,
    P_-(\lambda_q^{1/2}(H-\lambda_q))\} \leq \Tr \, \varrho (\lambda_q^{1/2}(H-\lambda_q))
    $$
    $$
    \leq \Tr \{\one_{[-R,R]}(\lambda_q^{1/2}(H-\lambda_q))\, P_+(\lambda_q^{1/2}(H-\lambda_q))\}.
    $$
    By \eqref{312} it follows that for all sufficiently large $q$,
\begin{multline}
\Tr \{\one_{(\lambda_q - B, \lambda_q + B)}(H) \,
P_-(\lambda_q^{1/2}(H-\lambda_q))\}
\leq
\Tr \varrho
(\lambda_q^{1/2}(H-\lambda_q))
\\
\leq
\Tr\{\one_{(\lambda_q - B, \lambda_q + B)}(H) \,
P_+(\lambda_q^{1/2}(H-\lambda_q))\}. \label{316}
\end{multline}

By Lemma~\ref{l31} and Theorem~\ref{th13}(iii),
we have
\begin{multline*}
\lim_{q\to\infty} \lambda_q^{-1/2}
\Tr\{\one_{(\lambda_q - B,
\lambda_q + B)}(H) \, P_\pm(\lambda_q^{1/2}(H-\lambda_q))\}
\\
=\frac{1}{2\pi} \int_{{\mathbb T}} \,
 \int_{\R}  \; P_\pm(\widetilde{V}(\omega, b)) \, db \, d \omega =
\int_{\re} P_\pm(t) d\mu(t).
\end{multline*}
Combining this with \eqref{316}, we get
    $$
    \limsup_{q \to \infty}  \lambda_q^{-1/2}
    \Tr \varrho (\lambda_q^{1/2}(H-\lambda_q))
    \leq \int_{\re} P_+(\lambda) d\mu (\lambda),
    $$
    $$
    \liminf_{q \to \infty}  \lambda_q^{-1/2}
    \Tr  \varrho (\lambda_q^{1/2}(H-\lambda_q))
    \geq \int_{\re} P_-(\lambda) d\mu (\lambda).
    $$
Finally,    by \eqref{315},
$$
\int_{\re} (P_+(\lambda) - P_-(\lambda)) d\mu (\lambda) \leq
\varepsilon \int_{\re} \lambda^{\ell_0} d\mu (\lambda).
$$
By \eqref{a9}, the integral in the r.h.s. is finite.
Since $\varepsilon>0$ can be taken arbitrary small, we obtain the
required statement.
\qed

\appendix
\section{}
\subsection{Proof of formula (3.6)}

    By definition,
    $$
    \widehat{\Psi_q}(\zeta) = \frac{2(-1)^q}{(2\pi)^2} \int_{\rd}
    e^{-iz\zeta} {\rm L}_q(2|z|^2) e^{-|z|^2} dz =
\frac{(-1)^q}{(2\pi)^2} \int_{\rd}
    e^{-iu\zeta/\sqrt{2}} {\rm L}_q(|u|^2) e^{-|u|^2/2} du
    $$
    (see \eqref{17a}). Further, by \cite[Eq. 22.12.6]{abst} we
    have
    $$
    {\rm L}_q(|u|^2) = {\rm L}_q(u_1^2 + u_2^2) = \sum_{m = 0}^q
    {\rm L}_m^{(-1/2)}(u_1^2) {\rm L}_{q-m}^{(-1/2)}(u_2^2), \quad u
    \in \rd,
    $$
    where ${\rm L}_m^{(-1/2)}$, $m \in {\mathbb Z}_+$, are the generalized
    Laguerre
    polynomials of order $-1/2$. By \cite[Eq. 22.5.38]{abst}
    we have
    $$
    {\rm L}_m^{(-1/2)}(t^2) = \frac{(-1)^m}{m! 2^{2m}} {\rm
    H}_{2m}(t), \quad t \in \re, \quad m \in {\mathbb Z}_+,
    $$
    where ${\rm H}_m$ are the
    Hermite polynomials.
    Therefore,
    $$
    {\rm L}_q(|u|^2) = \frac{(-1)^q}{2^{2q}} \sum_{m = 0}^q \frac{
    {\rm H}_{2m}(u_1) {\rm H}_{2q - 2m}(u_2)}{m! (q-m)!},
    $$
    and
    $$
    \widehat{\Psi_q}(\zeta) =
    $$
    $$
    \frac{1}{2^{2q+2}\pi^2} \sum_{m = 0}^q \frac{
    1}{m! (q-m)!} \int_\re e^{-iu_1\zeta_1/\sqrt{2}} {\rm
    H}_{2m}(u_1)e^{-u_1^2/2} du_1 \int_\re e^{-iu_2\zeta_2/\sqrt{2}} {\rm
    H}_{2q-2m}(u_2)e^{-u_2^2/2} du_2.
    $$
    It is well known that the functions ${\rm
    H}_{m}(t)e^{-t^2/2}$, $t \in \re$, $m \in {\mathbb Z}_+$, are
    eigenfunctions of the unitary Fourier transform with
    eigenvalues equal to $(-i)^m$ (see e.g. \cite{birsol}). Hence,
    $$
    \widehat{\Psi_q}(\zeta) = \frac{(-1)^q}{2^{2q+1}\pi}  \sum_{m = 0}^q \frac{
    1}{m! (q-m)!} {\rm H}_{2m}(2^{-1/2} \zeta_1) {\rm H}_{2q - 2m}(2^{-1/2}
    \zeta_2)e^{-|\zeta|^2/4} =
    $$
    $$
     (2\pi)^{-1}{\rm L}_q(2^{-1}|\zeta|^2)
    e^{-|\zeta|^2/4} = (-1)^q \Psi_q(2^{-1} \zeta)/2.
    $$

\subsection{Proof of  estimate (3.10)}
Denote $u_q(x)=e^{-x/2}\Lag_q(x)$, $v_q(x)=J_0(\sqrt{(4q+2)x})$.
Using the differential equations for the Laguerre polynomials and
for the Bessel functions, one easily checks that $u_q$ and $v_q$
satisfy
\begin{align*}
xu''_q(x)+u'_q(x)+(q+\tfrac12)u_q(x)&=\tfrac{x}{4}u_q(x),
\\
xv''_q(x)+v'_q(x)+(q+\tfrac12)v_q(x)&=0.
\end{align*}
Using these differential equations and the initial conditions
for $u_q(x)$, $v_q(x)$ at $x=0$, it is easy to verify that
$u_q$ satisfies the integral equation
\begin{gather}
u_q=v_q+K_qu_q,
\quad
(K_q f)(x)=\int_0^x F_q(x,y)f(y)dy,
\label{app1}
\\
F_q(x,y)=-\frac\pi4 y
\bigl(J_0(\sqrt{(4q+2)x})Y_0(\sqrt{(4q+2)y})-Y_0(\sqrt{(4q+2)x})J_0(\sqrt{(4q+2)y}\bigr).
\notag
\end{gather}
This argument is borrowed from \cite{suetin}. Iterating \eqref{app1}, we obtain
\begin{equation}
u_q-v_q=K_qv_q+K_q^2u_q.
\label{app2}
\end{equation}
Now it remains to estimate the two terms in the r.h.s. of \eqref{app2}
in an appropriate way. Using the estimates
$$
\abs{J_0(x)}\leq C/\sqrt{x},
\qquad
\abs{Y_0(x)}\leq C/\sqrt{x},
\qquad x>0,
$$
we obtain
$$
\abs{F_q(x,y)}\leq C q^{-1/2}x^{-1/4}y^{3/4},
\qquad
q\in\N, \quad x>0.
$$
This yields
\begin{equation}
\Abs{\int_0^x F_q(x,y)v_q(y)dy} \leq C q^{-3/4} x^{-1/4}\int_0^x
y^{1/2}dy =Cq^{-3/4}x^{5/4}. \label{app5}
\end{equation}
Next, using the estimate
$\abs{u_q(x)}\leq 1$
(see \cite[Eq. 22.14.12]{abst}),
we obtain
$$
\abs{(K_qu_q)(x)}
\leq
Cq^{-1/2} x^{-1/4}\int_0^x y^{3/4}dy
=
Cq^{-1/2}x^{3/2},
$$
and so
\begin{equation}
\abs{(K_q^2u_q)(x)}
\leq
Cq^{-1}x^{-1/4}\int_0^x y^{\frac32+\frac34}dy
=
Cq^{-1}x^3.
\label{app6}
\end{equation}
Combining \eqref{app2} with \eqref{app5} and \eqref{app6},
we obtain the required estimate \eqref{b7}.\\

{\bf Acknowledgements.} A. Pushnitski and G. Raikov were partially
supported by  {\em N\'ucleo Cient\'ifico ICM} P07-027-F ``{\em
Mathematical Theory of Quantum and Classical Magnetic Systems"}.
C. Villegas-Blas was partially supported by the same {\em N\'ucleo
Cient\'ifico} within the framework of the {\em International
Spectral Network}, and by PAPIIT-UNAM 109610-2. G. Raikov was
partially supported by the UNAM, Cuernavaca, during his stay in
2008, and by the Chilean Science
Foundation {\em Fondecyt} under Grant 1090467. \\
The authors are grateful for hospitality and financial
support to the Bernoulli Center, EPFL, Lausanne, where this work
was initiated within the framework of the Program ``{\em Spectral
and Dynamical Properties of Quantum Hamiltonians}", January - June
2010.

{\sc Alexander Pushnitski}\\
 Department of Mathematics,\\
 King's College London,\\
Strand, London, WC2R 2LS, United Kingdom\\
E-mail: alexander.pushnitski@kcl.ac.uk\\

{\sc Georgi Raikov}\\
 Facultad de
Matem\'aticas,\\ Pontificia Universidad Cat\'olica de Chile,\\
Vicu\~na Mackenna 4860, Santiago de Chile\\
E-mail: graikov@mat.puc.cl\\

{\sc Carlos Villegas-Blas}\\
 Instituto de Matem\'aticas,\\
 Universidad Nacional Auton\'oma de M\'exico, \\
 Cuernavaca, Mexico\\
E-mail: villegas@matcuer.unam.mx
\end{document}